\documentclass[12pt,english,a4paper]{article}
\usepackage{babel}
\usepackage[latin1]{inputenc}
\usepackage{tikz}
\usepackage{amssymb}
\usepackage{amsmath}
\usepackage{amsthm}
\usepackage{hyperref}
\usepackage{geometry}
\usepackage{mathtools}
\newtheorem{theorem}{Theorem}
\newtheorem{corollary}[theorem]{Corollary}
\newtheorem{proposition}[theorem]{Proposition}
\newtheorem{lemma}[theorem]{Lemma}
\newtheorem{example}[theorem]{Example}

\newtheorem{remark}[theorem]{Remark}
\newtheorem{definition}[theorem]{Definition}
\newtheorem{conjecture}[theorem]{Conjecture}

\newcommand{\cF}{\mathcal{F}}
\newcommand{\cG}{\mathcal{G}}
\newcommand{\cI}{\mathcal{I}}

\newcommand{\cP}{\mathcal{P}}
\newcommand{\cQ}{\mathcal{Q}}

\newcommand{\C}{\mathbb{C}}
\newcommand{\N}{\mathbb{N}}
\newcommand{\Q}{\mathbb{Q}}
\newcommand{\R}{\mathbb{R}}
\newcommand{\Z}{\mathbb{Z}} 

\newcommand\rit{%
    \tikz[line cap=round,x=1ex,y=1ex,line width=0.3pt]
    {\draw (0,0)--(0,1)--(1,0)--(0,0);}%
}

\begin{document}
  \title{Plane point sets with many squares or isosceles right triangles}
  \date{}
  \author{
    %%\begin{minipage}[t]{60mm}
     {\small \sc Sascha Kurz (sascha.kurz@uni-bayreuth.de)}\thanks{This paper is inspired by an extensive discussion on the Sequence Fans Mailing List with contributions from many people.}\\ 
     {\footnotesize Department of Mathematics, University of Bayreuth, D-95440 Bayreuth, Germany}
    %%\end{minipage}
    %%\begin{minipage}[t]{60mm}
    %%{\small \sc ???}\\%\thanks{The authors thank Peter Biryukov and Juris Steprans for valuable comments and discussions.}\\
    %%{\footnotesize ???}\\
    %%{\footnotesize ???}\\
    %%{\footnotesize ???}\\
    %%{\footnotesize ???}
    %%\end{minipage}
  }
  \maketitle
  
  \vspace*{-6mm}
  
  \small{
  \noindent
  $$
  \begin{array}{ll}
    \text{\textbf{Keywords:}} & \text{Erd\H{o}os problems, similar squares, isosceles right triangles, Euclidean plane} \\
    \text{\textbf{MSC:}} & \text{52C10 (05C35)} \\
  \end{array}
  $$
  }
  \noindent
  \rule{\textwidth}{0.3 mm}
  %%\begin{abstract}
    \noindent
    \textbf{Abstract:} 
    How many squares are spanned by $n$ points in the plane? Here we study the corresponding maximum possible number $S_{\square}(n)$ of squares and determine the exact values for all 
    $n\le 17$. For $18\le n\le 100$ we give lower bounds for $S_{\square}(n)$. Besides that a few preliminary structural results are obtained. 
    For the related problem of the maximum possible number $S_{\rit}(n)$ of isosceles right triangles we determine the exact values for $n\le 14$ and give lower bounds for 
    $15\le n\le 50$.
  %%\end{abstract}
  %%\vspace*{-5mm}
  \noindent
  \rule{\textwidth}{0.3 mm}

\section{Introduction}
\label{sec_introduction}
Given a finite set of points $\cQ\subset\R^2$, what is the maximum number $S_{\cQ}(n)$ of similar copies that can be contained in an $n$-point set in the plane? 
The origin of this problem can be traced back at least to Erd\H{o}s and Purdy \cite{erdHos1976some}. Besides being now a classical problem in combinatorial geometry 
there are connections to pattern recognition problems, see e.g.\ \cite{brass2002combinatorial} and the references cited therein. In such applications 
similarity is mostly replaced by congruency, so that we denote the maximum number of congruent copies of a finite set $\cQ\subset\R^2$ that can be contained in an 
$n$-point set in the plane by $C_{\cQ}(n)$. An easy algorithm for the corresponding congruent subset detection problem of $\cQ$ in $\cP$ is to choose two 
arbitrary (different) points $q_1,q_2\in\cQ$ and to loop over all pairs of points $p_1,p_2\in\cP$ where $d(q_1,q_2)=d(p_1,p_2)$, i.e., where the distances coincide. 
The complexity analysis of the algorithm requires one to determine $C_{\left\{q_1,q_2\right\}}(n)$, i.e., the maximum number of unit distances in the plane (sequence 
A186705\footnote{For the subsequently mentioned integer sequences see the {\lq\lq}On-Line Encyclopedia of Integer Sequences{\rq\rq} (OEIS) at \url{https://oeis.org}.}), a 
famous open problem introduced by Erd\H{o}s \cite{erdos1946sets}. The best upper bound known still is $O\!\left(n^{4/3}\right)$ \cite{spencer1984unit}, with a 
recent constant factor improvement, see \cite{agoston2020improved}.\footnote{This upper bound applies to all strictly convex norms, not just Euclidean distance, and can in 
fact be attained for certain special norms, see \cite{valtr2005strictly,swanepoel2018combinatorial}.} From that we can conclude $C_{\cQ}(n)\in O\!\left(n^{4/3}\right)$ 
and all congruent copies of $\cQ$ can be found in $O\!\left(\#\cQ\cdot n^{4/3}\log n\right)$ time, see \cite{brass2002combinatorial} for more details. 
For $S_{\cQ}(n)$ we have an upper bound of $n(n-1)$ and quadratic lower bound if $\cQ$ contains only algebraic points \cite{elekes1994similar}. In \cite{laczkovich1997number} 
a characterization of the point sets $\cQ$ with $S_{\cQ}(n)=\Theta(n^2)$ was obtained. All similar copies of $\cQ$ can be found in $O\!\left(\#\cQ\cdot n^{2}\log n\right)$ time, 
see e.g.\ \cite{brass2002combinatorial}. However, the existence of $\lim\limits_{n\to\infty} S_{\cQ}(n)$ is unknown 
for all non-trivial sets $\cQ$. Here we want to study the maximum number $S_{\square}(n)$ of squares contained in an $n$-point set in the plane (sequence A051602). We will 
be mainly interested in the determination of exact values or tight bounds for $S_{\square}(n)$ for the cases where $n$ is rather small. Taking the points of 
an $m\times m$ integer grid gives  $\liminf\limits_{n\to\infty} \frac{S_{\square}(n)}{n^2}\ge\frac{1}{12}$ (see sequence A002415 for the precise counts of squares). 
Taking the integer points inside circles of increasing radii gives $\liminf\limits_{n\to\infty} \frac{S_{\square}(n)}{n^2}\ge \frac{1-\frac{2}{\pi}}{4}> \frac{1}{11.008}$, 
see the comments in sequence A051602.\footnote{https://oeis.org/A051602/a051602\_2.txt} 
Denote the maximum number of isosceles right triangles in an $n$-point set by $S_{\rit}(n)$ and observe that each square consists of four such triangles. With this, 
the upper bound $S_{\rit}(n)\le \left\lfloor \tfrac{2}{3}(n-1)^2-\tfrac{5}{3}  \right\rfloor$ from \cite{abrego2011maximum} gives $\limsup\limits_{n\to\infty} 
\frac{S_{\square}(n)}{n^2}\le\frac{1}{6}$. In Proposition~\ref{proposition_general_upper_bound}  we will show 
$\limsup\limits_{n\to\infty} \frac{S_{\square}(n)}{n^2}\le\frac{1}{8}$. For the maximum number of equilateral triangles in the plane we refer to 
\cite{abrego2000maximum,pach2003many}. The latter reference takes the additional assumption that the $n$ points must be in convex position, see 
also \cite{xu2004number} for isosceles right triangles under the same assumption. 

Of course similar problems can be considered in higher dimensions or with different metrics \cite{brass2005problems}. Another variant is the number of rhombi or parallelograms 
contained in a plane point set. Since not all rhombi or parallelograms are similar the upper bound $O(n^2)$ does not apply and is indeed violated. 
For e.g.\ axis-parallel squares the lower bound $\Omega(n^2)$ is violated, cf.~\cite[Theorem 1]{van1991finding}. For results on repeated angles 
see e.g.\ \cite{pach1992repeated}. Another interesting variant is the maximum number of unit circles spanned by $n$ points in the plane, see e.g.\ \cite{harborth1986point}.

The remaining part is structured as follows. Section~\ref{sec_rit} is devoted to point sets with many isosceles right triangles. The approach of 
constructing point sets by recursively adding the vertices of an additional square is studied in depth in Section~\ref{sec_rec_extension}. For every point set in $\R^2$ there exists 
another point set in $\Z^2$ that spans at least as many squares, see Theorem~\ref{thm_grid_assumption}. An explicit upper bound on the necessary grid size is stated in 
Proposition~\ref{prop_grid_size_bound}. While there are infinitely many $7$-point sets spanning two squares that are pairwise non-similar, for all numbers $n$ of points and 
$m$ of squares there exist only finitely many equivalence classes if one uses a suitable combinatorial description, see Definition~\ref{definition_oriented_square_set} and 
Definition~\ref{definition_oriented_square_set_realizable}. With this, the determination of $S_{\square}(n)$ becomes a finite computational problem. In 
Section~\ref{sec_bounds_squares} we discuss bounds for $S_{\square}(n)$ and determine the exact values for all $n\le 17$. In an appendix we list several point sets that provide 
lower bounds for different variants of the problem.

\section{Plane point sets with many isosceles right triangles}
\label{sec_rit}
By definition an $n$-point set $\cP$ is a subset of $\R^2$ of cardinality $\#\cP=n$. By $S_{\rit}(\cP)$ we denote the number of isosceles right triangles contained 
in $\cP\subseteq \R^2$. In \cite[Theorem 6]{abrego2011maximum} the numbers $S_{\rit}(3)=1$, $S_{\rit}(4)=4$, $S_{\rit}(5)=8$, $S_{\rit}(6)=11$, $S_{\rit}(7)=15$, 
$S_{\rit}(8)=20$, and $S_{\rit}(9)=28$ (sequence A186926) were determined. Each pair of points can form the edge of six different isosceles right triangles, see 
Figure~\ref{fig_1_extension_line}, so that 
$S_{\rit}(n)\le {n\choose 2}\cdot 6/3=n(n-1)$, which equals the general upper bound $S_{\cQ}(n)\le n(n-1)$. As mentioned, the improved upper bound 
$S_{\rit}(n)\le \left\lfloor \tfrac{2}{3}(n-1)^2-\tfrac{5}{3}  \right\rfloor$ for $n\ge 3$ was shown in \cite[Theorem 5]{abrego2011maximum}. 

\begin{figure}[htp!]
\begin{center}
\begin{tikzpicture}[scale=0.6]
  \draw[thick,blue](0,2)--(2,0);  
  \draw[thick,blue](0,2)--(0,-2);
  \draw[thick,blue](2,2)--(2,-2);
  \draw[thick,blue](0,-2)--(2,0);
  \draw[thick,blue](2,2)--(0,0);
  \draw[thick,blue](2,-2)--(0,0);
  \fill[black] (0,0)circle(4.5pt);
  \fill[black] (2,0)circle(4.5pt);
  \fill[black] (2,2)circle(4.5pt);
  \fill[black] (0,2)circle(4.5pt);
  \fill[black] (2,-2)circle(4.5pt);
  \fill[black] (0,-2)circle(4.5pt);
  \fill[black] (1,1)circle(4.5pt);
  \fill[black] (1,-1)circle(4.5pt);
  \draw[thick](0,0)--(2,0);
  \draw (0,0) node[black,anchor=east]{$p_1$};
  \draw (2,0) node[black,anchor=west]{$p_2$};
  \draw (0,2) node[black,anchor=north east]{$p_3^1$};
  \draw (2,2) node[black,anchor=north west]{$p_3^2$};
  \draw (0,-2) node[black,anchor=south east]{$p_3^3$};
  \draw (2,-2) node[black,anchor=south west]{$p_3^4$};
  \draw (1,1) node[black,anchor=south]{$p_3^5$};
  \draw (1,-1) node[black,anchor=north]{$p_3^6$};
\end{tikzpicture}
\caption{The third vertices of the six isosceles right triangles containing a given line segment $\left\{p_1,p_2\right\}$ as an edge.}
\label{fig_1_extension_line}
\end{center}
\end{figure}
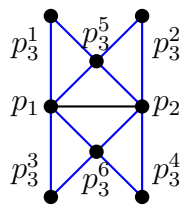

If $\cP'$ arises from 
an $n$-point set $\cP$ by picking a pair $\left\{p_1,p_2\right\}$ of points in $\cP$ and setting $\cP'=\cP\cup\left\{p_3\right\}$ for one of the six possibilities 
for $p_3$ such that $\left\{p_1,p_2,p_3\right\}$ forms an isosceles right triangle, then we say that $\cP'$ is obtained by $1$-extension. If an $n$-point set $\cP$ is 
obtained by a sequence of $1$-extensions starting from an arbitrary isosceles right triangle, then we say that $\cP$ can be obtained by $1$-extension. In 
Table~\ref{table_iso_types_exhaustive_1_extension_triangles} we have listed the number of non-similar points sets that can be obtained by $1$-extension per number $n$ of points 
and $m$ of isosceles right triangles. The number of isosceles right triangles of such an $n$-point set is at least $n-2$. E.g.\ a $5$-point set obtained by joining two isosceles right triangles 
at a common vertex, such that no further isosceles right triangles are spanned, cannot be obtained by $1$-extension. So, clearly not all $n$-point sets can be obtained 
by $1$-extension. In Proposition~\ref{proposition_1_extension_obtainable} we will show that for points sets with at most $7$ points a sufficiently large 
number of isosceles right triangles implies that the point set can be obtained by $1$-extension. For $n=9$ and $7\le m\le 8$ the numbers of cases are quite numerous so that 
we did not have determined the number of similarity types. 

\begin{table}[htp!]
  \begin{center}
    \begin{tabular}{llllllllllllllllllll}
      \hline
      $n$ & 3 & 4 & 4 & 4 & 5  & 5 & 5 & 5 & 5 & 5 & 6   & 6  & 6  & 6  & 6 & 6 & 6  & 6  \\ 
      $m$ & 1 & 2 & 3 & 4 & 3  & 4 & 5 & 6 & 7 & 8 & 4   & 5  & 6  & 7  & 8 & 9 & 10 & 11 \\ 
      \# &  1 & 2 & 1 & 1 & 16 & 4 & 2 & 1 & 0 & 1 & 232 & 88 & 38 & 16 & 6 & 1 & 3  & 1  \\
      \hline
    \end{tabular}  
    \begin{tabular}{llllllllllllllll}
      \hline
      $n$ & 7    & 7    & 7    & 7   & 7   & 7  & 7  & 7  & 7  & 7  & 7  & 8      \\ 
      $m$ & 5    & 6    & 7    & 8   & 9   & 10 & 11 & 12 & 13 & 14 & 15 & 6      \\ 
      \#  & 5383 & 2397 & 1051 & 490 & 164 & 50 & 39 & 17 & 7  & 6  & 2  & 172408 \\ 
      \hline
    \end{tabular}
    \begin{tabular}{llllllllllllllll}
      \hline
      $n$ & 8     & 8     & 8     & 8    & 8    & 8    & 8   & 8   & 8   & 8  & 8  & 8  \\
      $m$ & 7     & 8     & 9     & 10   & 11   & 12   & 13  & 14  & 15  & 16 & 17 & 18 \\ 
      \#  & 89266 & 41475 & 19925 & 7123 & 2488 & 1513 & 685 & 253 & 137 & 75 & 31 & 17 \\ 
      \hline
    \end{tabular}
    \begin{tabular}{llllllllllllllll}
      \hline
      $n$ & 8  & 8  & 9       & 9      & 9      & 9      & 9     & 9     & 9     & 9    \\
      $m$ & 19 & 20 & 9       & 10     & 11     & 12     & 13    & 14    & 15    & 16   \\ 
      \#  & 2  & 5  & 1623291 & 878770 & 379869 & 142722 & 77106 & 36226 & 14662 & 7194 \\
      \hline
    \end{tabular}
    \begin{tabular}{llllllllllllllll}
      \hline
      9    & 9    & 9   & 9   & 9   & 9  & 9  & 9  & 9  & 9  & 9  & 9  \\
      17   & 18   & 19  & 20  & 21  & 22 & 23 & 24 & 25 & 26 & 27 & 28 \\ 
      3475 & 1474 & 856 & 310 & 186 & 94 & 33 & 20 & 5  & 4  & 1  & 1  \\
      \hline
    \end{tabular}  
    \caption{Number of non-similar points sets $\cP$ produced by recursive $1$-extension per number of points $n$ and
    isosceles right triangles $m$.}
    \label{table_iso_types_exhaustive_1_extension_triangles}
  \end{center}
\end{table}

\begin{definition}
  Let $\cP\subset \R^2$ be an arbitrary point set. A point set $\cP'\subseteq \cP$ is called a \emph{$1$-extension subconfiguration} if $\cP'\subseteq \cP$ can be obtained by 
  $1$-extension. If no proper superset of $\cP'$ can be obtained by $1$-extension, then we call $\cP'$ $1$-extension maximal.
\end{definition}
\begin{lemma}
  \label{lemma_1_extension_maximal_rit}
  Let $\cP\subset\R^2$ be an arbitrary $n$-point set and $\cP'$ be a $1$-extension maximal subconfiguration. Then we have 
  $$
    S_{\rit}(\cP) \le S_{\rit}(\cP') + 2\cdot{{\# \cP\backslash\cP'} \choose 2}+S_{\rit}(\#\cP\backslash\cP')
    \le S_{\rit}(\#\cP') + 2\cdot{{\# \cP\backslash\cP'} \choose 2}+S_{\rit}(\#\cP\backslash\cP'). 
  $$
\end{lemma}
\begin{proof}
  Let $t$ be an arbitrary isosceles right triangle in $\cP$ that is not contained in $\cP'$. Since $\cP'$ is $1$-extension maximal either $2$ or $3$ vertices of $t$ have to be contained 
  in $\cP\backslash\cP'$. The second case can occur at most $S_{\rit}(\#\cP\backslash\cP')$ times. For the first case we observe that there are 
  ${{\# \cP\backslash\cP'}\choose 2}$ pairs of points in $\cP\backslash\cP'$. Consider such a pair $\left\{p_1,p_2\right\}$ and the third point $p_3^i$ of an isosceles right triangle 
  as in Figure~\ref{fig_1_extension_line}. Since the triples $\left\{p_3^1,p_3^2,p_1\right\}$, $\left\{p_3^3,p_3^4,p_1\right\}$, $\left\{p_3^1,p_3^3,p_2\right\}$, 
  $\left\{p_3^2,p_3^4,p_1\right\}$, $\left\{p_3^5,p_3^6,p_1\right\}$, $\left\{p_3^1,p_3^5,p_1\right\}$, $\left\{p_3^2,p_3^5,p_2\right\}$, $\left\{p_3^3,p_3^6,p_1\right\}$, 
  and $\left\{p_3^4,p_3^6,p_2\right\}$ form isosceles right triangles with exactly two vertices in $\cP'$ at most two out of the six points $p_3^1,\dots,p_3^6$ can be 
  obtained in $\cP'$. Thus, the first case can occur at most $2\cdot{{\# \cP\backslash\cP'}\choose 2}$ times.
\end{proof}

\begin{example}
  \label{ex_non_1_extension}
  Consider the $6$-point set
  $$
    \cP=\left\{(0,0),(1,0),(0,1),(1,1),\left(\tfrac{3}{5},\tfrac{6}{5}\right),\left(\tfrac{2}{5},\tfrac{4}{5}\right)\right\}
  $$
  and its subset $\cP'=\left\{(0,0),(1,0),(0,1),(1,1)\right\}$ with $S_{\rit}(\cP)=6$ and $S_{\rit}(\cP')=4$. The $1$-extension subconfiguration $\cP'$ is $1$-extension maximal, 
  so that the upper bound of Lemma~\ref{lemma_1_extension_maximal_rit} is attained with equality. 
\end{example}

An easy averaging argument implies the existence of subconfigurations with relatively many isosceles right triangles:
\begin{lemma}(\cite[Lemma 8]{abrego2011maximum})\\
  \label{lemma_avg_t}
  Let $\cP$ be an $n$-point set. If $S_{\rit}(A)\le b$ for all $\cP'\le\cP$ with $\# \cP'=k$, then we have
  $$
    S_{\rit}(\cP)\le \left\lfloor {n\choose 3}\cdot b\,/{k\choose 3}\right\rfloor.
  $$
\end{lemma}
\begin{proof}
  The point set $\cP$ contains ${n\choose k}$ subsets of cardinality $k$. Since each isosceles right triangle is contained in ${{n-3}\choose{k-3}}$ such subsets of cardinality 
  $k$, we have 
  $$
    S_{\rit}(\cP)\le {n\choose k}\cdot b/ {{n-3}\choose{k-3}} = \frac{n(n-1)(n-2)}{k(k-1)(k-2)}\cdot b.
  $$ 
\end{proof}

Adding an arbitrary point to a $4$-point set consisting of four isosceles right triangles, we obtain a $5$-point set that consists of four isosceles right triangles in most cases. 
Such point sets obviously cannot be obtained by $1$-extension. In order to avoid these trivial modifications we forbid {\lq\lq}isolated{\rq\rq} points. More precisely, we call an 
$n$-point set $\cP$ \emph{reduced} if each point is contained in at least one isosceles right triangle.

\begin{lemma}
  \label{lemma_one_less}
  Let $\cP$ be a reduced $n$-point set and $\widetilde{\cP}$ be a $1$-extension maximal subconfiguration. Then, we have $\#\widetilde{\cP}\neq \#\cP$.
\end{lemma}
\begin{proof}
  Assume $\#\widetilde{\cP}=\#\cP$ and let $p$ be the unique vertex contained in $\cP\backslash\widetilde{\cP}$. Since $\cP$ is reduced there exist an isosceles right triangle 
  $t$ containing $p$ as a vertex. Since $t$ intersects $\widetilde{\cP}$ in two vertices, $\widetilde{\cP}$ is not $1$-extension maximal -- contradiction. 
\end{proof}

\begin{proposition}$\,$
  \label{proposition_1_extension_obtainable}
  \begin{enumerate}
    \item[(1)] Each reduced $3$-point set $\cP$ with $S_{\rit}(\cP)\ge 1$ can be obtained by $1$-extension and $S_{\rit}(3)=1$.\\[-7mm]
    \item[(2)] Each reduced $4$-point set $\cP$ with $S_{\rit}(\cP)\ge 2$ can be obtained by $1$-extension and $S_{\rit}(4)=4$.\\[-7mm]
    \item[(3)] Each reduced $5$-point set $\cP$ with $S_{\rit}(\cP)\ge 3$ can be obtained by $1$-extension and $S_{\rit}(5)=8$.\\[-7mm]
    \item[(4)] Each reduced $6$-point set $\cP$ with $S_{\rit}(\cP)\ge 7$ can be obtained by $1$-extension and $S_{\rit}(6)=11$.\\[-7mm]
    \item[(5)] Each reduced $7$-point set $\cP$ with $S_{\rit}(\cP)\ge 11$ can be obtained by $1$-ext.\ and $S_{\rit}(7)=15$.
  \end{enumerate}
\end{proposition}
\begin{proof}
  Given the data of Table~\ref{table_iso_types_exhaustive_1_extension_triangles} it suffices to show that $\cP$ can be obtained by $1$-extension 
  if $S_{\rit}(\cP)$ is at least as large as proposed. 

  The statement is trivial for $\#\cP=3$. If $\#\cP=4$, then consider an arbitrary isosceles right triangle $t$. The statement follows from the 
  fact that each additional isosceles right triangle $t'$ shares an edge with $t$. For $\#\cP=5$ we apply Lemma~\ref{lemma_avg_t} to conclude the existence of 
  a $4$-point set $\cP'$ with $S_{\rit}(\cP')\ge 2$. Since $S_{\rit}(3)=1<2$, $\cP'$ is reduced, so that $\cP'$ can by obtained by $1$-extension due to (2). 
  Now consider an isosceles right triangle $t$ with vertex $\cP\backslash\cP'$. Since $t$ intersects $\cP'$ in two points, also $\cP$ can be obtained by $1$-extension, 
  see Lemma~\ref{lemma_one_less}. For $\#\cP=6$ we apply Lemma~\ref{lemma_avg_t} to conclude the existence of a $5$-point set $\cP'$ with $S_{\rit}(\cP')\ge 4$. If $\cP'$ 
  is reduced, then $\cP'$ 
  can be obtained by $1$-extension due to (3) and an isosceles right triangle $t$ with vertex $\cP\backslash\cP'$ intersects $\cP'$ in two vertices, so that also 
  $\cP$ can be obtained by $1$-extension. If $\cP'$ is not reduced, then $S_{\rit}(3)=1$ implies the existence of a reduced $4$-point set $\cP''$ with 
  $S_{\rit}(\cP'')=4$ and (2) yields that $\cP''$ can be obtained by $1$-extension. Now let $\widetilde{\cP}$ be a $1$-extension maximal subconfiguration of $\cP$, 
  so that $\#\widetilde{\cP}\ge \#\cP''=4$. If $\#\widetilde{\cP}=6$, then $\cP$ can be obtained by $1$-extension. If $\#\widetilde{\cP}=5$, then we can again 
  consider an isosceles right triangle containing the missing sixth point as a vertex to conclude that $\cP$ can be obtained by $1$-extension. Thus, it 
  remains to consider the case $\#\widetilde{\cP}=4$, where Lemma~\ref{lemma_1_extension_maximal_rit} yields
  $$
    S_{\rit}(\cP)\le S_{\rit}(4)+2{2\choose 2}+S_{\rit}(2)=6<7,
  $$    
  which is a contradiction.
  
  Now let $\#\cP=7$ and $\cP'$ be a $1$-extension maximal subconfiguration of $\cP$ with $\#\cP'<7=\#\cP$. If $\#\cP'\in\{5,6\}$, then Lemma~\ref{lemma_1_extension_maximal_rit} 
  gives 
  $$
    S_{\rit}(\cP)\le \max\left\{ S_{\rit}(5)+2{2\choose 2}+S_{\rit}(2),S_{\rit}(6)+2{1\choose 2}+S_{\rit}(1)\right\}=10<11,
  $$    
  which is a contradiction. If no pair of points is contained in at least two isosceles right triangles, then $S_{\rit}(\cP)\le {\#\cP\choose 2}/3$, so that 
  we can conclude $\#\cP'=4$ in our situation. Lemma~\ref{lemma_1_extension_maximal_rit} then yields 
  $$
    S_{\rit}(\cP)\le S_{\rit}(4)+2{3\choose 2}+S_{\rit}(3)=11.
  $$   
  However, if the upper bound of Lemma~\ref{lemma_1_extension_maximal_rit} is attained, then the three points in $\cP\backslash\cP'$ form an isosceles right triangle 
  and each edge is contained in two further isosceles right triangles, which yields a subconfiguration of cardinality $5$ that can be obtained by $1$-extension -- 
  contradiction. 
\end{proof}
Similar as in Example~\ref{ex_non_1_extension}, we can start from the (up to similarity) unique $5$-point set $\cP'$ with $S_{\rit}(\cP')=8$ and add two further points 
to obtain an $7$-point set $\cP$ with $S_{\rit}(\cP)=10$ that cannot be obtained by $1$-extension.

\begin{lemma}
  \label{lemma_ilp_triangle}
  The maximum target value of the following integer linear program (ILP) is an upper bound for $S_{\rit}(n)$, where $N=\{1,\dots,n\}$:
  \begin{eqnarray*}
    \max \sum_{S\subseteq N:\#S =3} x_S\quad\text{subject to}\\
    \sum_{S\subseteq T:\#S=3} x_S \le S_{\rit}(\# T)\quad\forall T\subseteq N: 4\le \#T\le n-1\\
    -3y_{4,T}+\sum_{S\subseteq T:\#S=3} x_S \le 1 \quad\forall T\subseteq N: \#T =4\\
    {{t\!-\!1}\choose 2} \left(y_{t\!-\!1,R}-y_{t,T}\right)+\!\!\!\!\!\!\!\!\!\!\sum_{S\subseteq T:S\not\subseteq R,\#S=3}\!\!\!\!\!\!\!\!\!\! x_S \le {{t\!-\!1}\choose 2} \quad\forall R\subset T\subseteq N: \#T=t, \#R=t\!-\!1, 5\le t\le n\\
    x_S\in\{0,1\} \quad\forall S\subseteq N:\#S=3\\ 
    y_{\#T,T}\in\{0,1\} \quad\forall T\subseteq N:4\le \#T\le n
  \end{eqnarray*}    
\end{lemma}
\begin{proof}
  Given an $n$-point set $\cP$ we set $x_S=1$ iff the three elements in $S$ form the vertices of an isosceles right triangle in $\cP$. For each $T\subseteq N$ with $4\le \#T\le n$ 
  we set $y_{\# T,T}=1$ iff the restriction of $\cP$ to $T$ can by obtained by $1$-extension. With this, the first set of inequalities is automatically satisfied. For 
  the second set of inequalities we remark that either at most one of the four $3$-sets $S\subset T$ forms the vertices of an isosceles right triangle or the restriction of 
  $\cP$ to $T$ (of cardinality $4$) can be obtained by $1$-extension since there are at least two isosceles right triangles intersecting in an edge. For the third set of inequalities 
  we observe that $\sum_{S\subseteq T:S\not\subseteq R,\#S=3} x_S \le {{t-1}\choose 2}$. So, if $y_{t\!-\!1,R}=0$, then the inequality is automatically satisfied. 
  The same is true if $y_{t\!-\!1,R}=1$ and $\sum_{S\subseteq T:S\not\subseteq R,\#S=3}=0$. If $y_{t\!-\!1,R}=1$ and $\sum_{S\subseteq T:S\not\subseteq R,\#S=3}\ge 1$, then 
  the restriction of $\cP$ to $R$ can be obtained by $1$-extension and there exists an isosceles right triangle that intersects $R$ in two vertices, so that 
  also the restriction of $\cP$ to $T$ can be obtained by $1$-extension and $y_{t,T}$. In the latter case the inequality is also satisfied. To sum up, choosing the $x_S$ and 
  $y_{\#T,T}$ as stated initially yields a feasible solution of the ILP with target value $\sum_{S\subseteq N:\#S =3} x_S=S_{\rit}(\cP)$. Thus, the optimal target value 
  of the ILP is at least as large as $S_{\rit}(n)$.       
\end{proof}
%% ilp_model_triangle

\begin{proposition}$\,$
  \label{proposition_1_extension_obtainable_2}
  \begin{enumerate}
    \item[(1)] Each reduced $8$-point set $\cP$ with $S_{\rit}(\cP)\ge 14$ can be obtained by $1$-ext.\ and $S_{\rit}(8)=20$.\\[-7mm]
    \item[(2)] Each reduced $9$-point set $\cP$ with $S_{\rit}(\cP)\ge 18$ can be obtained by $1$-ext.\ and $S_{\rit}(9)=28$.
  \end{enumerate}
\end{proposition}
\begin{proof}
  Given the data of Table~\ref{table_iso_types_exhaustive_1_extension_triangles} it suffices to show that $\cP$ can be obtained by $1$-extension 
  if $S_{\rit}(\cP)$ is at least as large as proposed.
  
  Assume that $\cP$ is a reduced $8$-point set that cannot be obtained by $1$-extension and satisfies $S_{\rit}(\cP)\ge 14$. Let $\widetilde{\cP}$ be a $1$-extension maximal 
  subconfiguration of $\cP$, so that Lemma~\ref{lemma_one_less} implies $\#\widetilde{\cP}\le 6$. If $\#\widetilde{\cP}=6$, then Lemma~\ref{lemma_1_extension_maximal_rit} 
  implies $S_{\rit}(\cP)\le S_{\rit}(6)+2\cdot{2\choose 2}+S_{\rit}(2)=13$, which is a contradiction. If $\#\widetilde{\cP}=5$, then Lemma~\ref{lemma_1_extension_maximal_rit} 
  implies $S_{\rit}(\cP)\le S_{\rit}(\widetilde{\cP})+2\cdot{3\choose 2}+S_{\rit}(3)=S_{\rit}(\widetilde{\cP})+7$, so that $S_{\rit}(\widetilde{\cP})\ge 7$. From 
  Proposition~\ref{proposition_1_extension_obtainable}.(3) and Table~\ref{table_iso_types_exhaustive_1_extension_triangles} we conclude $S_{\rit}(\widetilde{\cP})=8$. 
  Next we apply the ILP from Lemma~\ref{lemma_ilp_triangle} with $n=8$ and the additional inequalities $y_{\#T,T}=0$ for all $T\subseteq N$ with $\#T\in\{6,7,8\}$ and 
  $$
    -8y_{5,T}+\sum_{S\subseteq T:\#S=3} x_S \ge 0  
  $$
  for all $T\subseteq N$ with $\#T=5$. We can computationally check that the optimal target value of this modified ILP is $12$, so that statement (1) follows.  
  %% From Lemma~\ref{lemma_minimum_degree} we conclude $\delta_{\max}\ge 7$ for the maximum degree.
  
  Assume that $\cP$ is a reduced $9$-point set that cannot be obtained by $1$-extension and satisfies $S_{\rit}(\cP)\ge 19$. Let $\widetilde{\cP}$ be a $1$-extension maximal 
  subconfiguration of $\cP$, so that Lemma~\ref{lemma_one_less} implies $\#\widetilde{\cP}\le 7$. If $\#\widetilde{\cP}=7$, then Lemma~\ref{lemma_1_extension_maximal_rit} 
  implies $S_{\rit}(\cP)\le S_{\rit}(7)+2\cdot{2\choose 2}+S_{\rit}(2)=17$, which is a contradiction. If $\#\widetilde{\cP}=6$, then Lemma~\ref{lemma_1_extension_maximal_rit} 
  implies $S_{\rit}(\cP)\le S_{\rit}(\widetilde{\cP})+2\cdot{3\choose 2}+S_{\rit}(3)=S_{\rit}(\widetilde{\cP})+7$, so that $S_{\rit}(\widetilde{\cP})=11$. 
  %% From Lemma~\ref{lemma_minimum_degree} we conclude $\delta_{\max}\ge 7$ for the maximum degree.   
  Next we apply the ILP from Lemma~\ref{lemma_ilp_triangle} with $n=9$ and the additional inequalities $y_{\#T,T}=0$ for all $T\subseteq N$ with $\#T\in\{7,8,9\}$ and 
  $$
    -11y_{6,T}+\sum_{S\subseteq T:\#S=3} x_S \ge 0  
  $$
  for all $T\subseteq N$ with $\#T=6$.  
  We can computationally check that the optimal target value of this modified ILP is $17$, so that statement (2) follows. 
\end{proof}

\begin{proposition}
  Each $10$-point set $\cP$ with $S_{\rit}(\cP)\ge 25$ can be obtained by $1$-extension from a subconfiguration $\cP'$ with $\#\cP'\le 9$, $S_{\rit}(\cP')\ge 18$ that can be 
  obtained by $1$-extension. Moreover, we have $S_{\rit}(10)=35$. 
\end{proposition}
\begin{proof}
  From Lemma~\ref{lemma_avg_t} we conclude the existence of a subconfiguration $\cP''\subset\cP$ with $\#\cP''=9$ and $S_{\rit}(\cP'')\ge 18$. Now let $\cP'$ arise from 
  $\cP''$ by removing isolated vertices, so that $\cP'$ is reduced and $S_{\rit}(\cP)'\ge 18$. Proposition~\ref{proposition_1_extension_obtainable_2} then 
  yields that $\cP'$ can be obtained by $1$-extension and $\#\cP'\in\{8,9\}$.

  We recursively apply $1$-extension starting from all point sets $\cP'$ with $\#\cP'\in\{8,9\}$ and $S_{\rit}(\cP')\ge 18$ that can be obtained by $1$-extension till cardinality 
  $n=10$ is reached for $\cP$. This computation yields $S_{\rit}(10)\ge 35$ and none of  the constructed point sets strictly improves upon this lower bound.
  
  Now let $\widetilde{\cP}$ be a $1$-extension maximal subconfiguration of $\cP$ with $\cP'\subseteq \widetilde{\cP}$. From Lemma~\ref{lemma_1_extension_maximal_rit} and 
  Lemma~\ref{lemma_one_less} we conclude $\#\widetilde{\cP}=\#\cP$, so that $S_{\rit}(10)=35$.
\end{proof}
Using the same technique we prove the following four statements:
\begin{proposition}
  Each $11$-point set $\cP$ with $S_{\rit}(\cP)\ge 34$ can be obtained by $1$-extension from a subconfiguration $\cP'$ with $\#\cP'\le 10$, $S_{\rit}(\cP')\ge 25$ that 
  can be obtained by $1$-extension. Moreover, we have $S_{\rit}(11)=43$.
\end{proposition}
\begin{proposition}
  Each $12$-point set $\cP$ with $S_{\rit}(\cP)\ge 45$ can be obtained by $1$-extension from a subconfiguration $\cP'$ with $\#\cP'\le 11$, $S_{\rit}(\cP')\ge 34$ that 
  can be obtained by $1$-extension. Moreover, we have $S_{\rit}(12)=52$.
\end{proposition}
\begin{proposition}
  Each $13$-point set $\cP$ with $S_{\rit}(\cP)\ge 58$ can be obtained by $1$-extension from a subconfiguration $\cP'$ with $\#\cP'\le 12$, $S_{\rit}(\cP')\ge 45$ that 
  can be obtained by $1$-extension. Moreover, we have $S_{\rit}(13)=64$.
\end{proposition}
\begin{proposition}
  Each $14$-point set $\cP$ with $S_{\rit}(\cP)\ge 73$ can be obtained by $1$-extension from a subconfiguration $\cP'$ with $\#\cP'\le 13$, $S_{\rit}(\cP')\ge 58$ that 
  can be obtained by $1$-extension. Moreover, we have $S_{\rit}(14)=74$.
\end{proposition}
\begin{theorem}
  We have $S_{\rit}(10)=35$, $S_{\rit}(11)=43$, $S_{\rit}(12)=52$, $S_{\rit}(13)=64$, $S_{\rit}(14)=74$, and all similarity types of the corresponding extremal configurations 
  are listed in Appendix~\ref{appendix_dissimilar_blbk_point_sets_rit}.
\end{theorem}

While the bound of Lemma~\ref{lemma_avg_t} is too weak to continue in the same manner for $n>14$, we can recursively apply $1$-extension keeping the best 30000 point sets in each 
iteration only in order to obtain heuristic lower bounds for $S_{\rit}(n)$. These lower bounds are given in Table~\ref{table_lower_bounds_number_of_contained_irt} and the 
known similarity types of attaining configurations are listed in Appendix~\ref{appendix_dissimilar_blbk_point_sets_rit}. For $15\le n\le 25$ the lower bounds of 
\cite{abrego2011maximum} are matched.

\begin{table}[htp]
\begin{center}
\begin{tabular}{lllllllllllllllllll}
\hline
  $n$ &              15 & 16 & 17  & 18  & 19  & 20  & 21  & 22  & 23  & 24  & 25  & 26  & 27  \\
  $S_{\rit}(n)\ge$ & 85 & 97 & 112 & 124 & 139 & 156 & 176 & 192 & 210 & 229 & 252 & 271 & 291 \\ 
\hline
\end{tabular}
\begin{tabular}{lllllllllllllllllll}
\hline
  $n$ &              28  & 29  & 30  & 31  & 32  & 33  & 34  & 35  & 36  & 37  & 38  & 39  \\
  $S_{\rit}(n)\ge$ & 314 & 338 & 363 & 389 & 417 & 448 & 473 & 501 & 531 & 564 & 594 & 626 \\
\hline
\end{tabular}
\begin{tabular}{llllllllllllllllll}
\hline
  $n$ &              40  & 41  & 42  & 43  & 44  & 45  & 46  & 47  & 48  & 49   & 50   \\
  $S_{\rit}(n)\ge$ & 659 & 696 & 728 & 763 & 799 & 836 & 874 & 914 & 955 & 1000 & 1038 \\
\hline
\end{tabular}
\caption{Lower bounds for the number of isosceles right triangles in $n$-point sets.} %% with $15\le n\le 50$.}
\label{table_lower_bounds_number_of_contained_irt}
\end{center}
\end{table}

\section{Recursively extending point sets by squares}
\label{sec_rec_extension}

By $S_{\square}(\cP)$ we denote the number of squares contained 
in $\cP\subseteq \R^2$. We say that $\cP$ can be embedded on the (integer) grid if there exists a similar point set $\cP'$ with $\cP'\subset\Z^2$. In our 
context most of the point sets can be assumed to be finite subsets of the integer grid. Instead of listing coordinates, we may also give a graphical representation, 
see Figure~\ref{fig_6_configurations}. Here we have also depicted the squares contained in the point set. By convention $\cP_{n,m}^i$ will always denote an $n$-point 
set with $S_{\square}(\cP_{n,m}^i)=m$, where $i$ is an index distinguishing non-similar point sets. The set of points of a single (unit) square is denoted by $\cP_{4,1}^1$; 
for other examples see Figure~\ref{fig_6_configurations}. For easier reference we label the points from $1$ to $n$. 

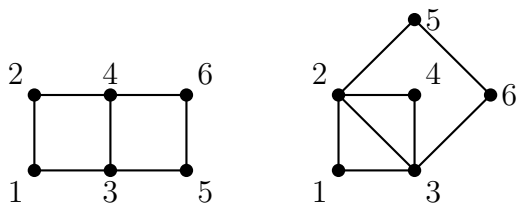
\begin{figure}[htp]
\begin{center}
\begin{tikzpicture}
  \draw (0,0) node[anchor=north east]{$1$};
  \draw (0,1) node[anchor=south east]{$2$};
  \draw (1,0) node[anchor=north]{$3$};
  \draw (1,1) node[anchor=south]{$4$};
  \draw (2,0) node[anchor=north west]{$5$};
  \draw (2,1) node[anchor=south west]{$6$};
  \fill[black] (0,0)circle(2.5pt);
  \fill[black] (0,1)circle(2.5pt);
  \fill[black] (1,0)circle(2.5pt);
  \fill[black] (1,1)circle(2.5pt);
  \fill[black] (2,0)circle(2.5pt);
  \fill[black] (2,1)circle(2.5pt);
  \draw[thick](0,0)--(2,0);
  \draw[thick](0,1)--(2,1);
  \draw[thick](0,0)--(0,1);  
  \draw[thick](1,0)--(1,1);
  \draw[thick](2,0)--(2,1);
  \draw (4,0) node[anchor=north east]{$1$};
  \draw (4,1) node[anchor=south east]{$2$};
  \draw (5,0) node[anchor=north west]{$3$};
  \draw (5,1) node[anchor=south west]{$4$};
  \draw (5,2) node[anchor=west]{$5$}; 
  \draw (6,1) node[anchor=west]{$6$};
  \fill[black] (4,0)circle(2.5pt);
  \fill[black] (4,1)circle(2.5pt);
  \fill[black] (5,0)circle(2.5pt);
  \fill[black] (5,1)circle(2.5pt);
  \fill[black] (5,2)circle(2.5pt);
  \fill[black] (6,1)circle(2.5pt);
  \draw[thick](4,0)--(5,0);
  \draw[thick](4,0)--(4,1);
  \draw[thick](4,1)--(5,1);  
  \draw[thick](5,0)--(5,1);
  \draw[thick](4,1)--(5,0);
  \draw[thick](5,2)--(6,1);
  \draw[thick](4,1)--(5,2);
  \draw[thick](5,0)--(6,1);
\end{tikzpicture}
\caption{Two non-similar point sets $\cP_{6,2}^1$ and $\cP_{6,2}^2$ consisting of $6$ points and $2$ squares.}
\label{fig_6_configurations}
\end{center}
\end{figure}

Since each square consists of four (corner) points, i.e.\ vertices, we have $S_{\square}(i)=0$ for $i\le 3$ and $S_{\square}(4)=1$. 
Let us now consider how the vertices of two different squares can overlap. To this end we distinguish between the four edges and the 
two diagonals of a square. In $\cP_{6,2}^2$ in Figure~\ref{fig_6_configurations} the vertices $2$ and $3$ form a diagonal of the square 
with vertices in $\{2,3,5,6\}$ as well as an edge of the square with vertices in $\{1,2,3,4\}$.
\begin{lemma}
  \label{lemma_two_points_determine_coordinates}
  Let $P_1,P_2\in\R^2$ be two arbitrary distinct points. Then there exist three different choices for pairs of points 
  $\left\{P_3,P_4\right\}$ such that $\left\{P_1,P_2,P_3,P_4\right\}$ spans a square. More precisely, denoting the coordinates 
  of $P_i$ by $\left(x_i,y_i\right)$ for $1\le i\le 4$, we have (up to a permutation of $P_3$ and $P_4$) that
  \begin{enumerate}
    \item[(a)] $\left(x_3,y_3\right)=\left(\frac{x_1+x_2+y_2-y_1}{2},\frac{y_1+y_2+x_1-x_2}{2}\right)$,   
               $\left(x_4,y_4\right)=\left(\frac{x_1+x_2+y_1-y_2}{2},\frac{y_1+y_2+x_2-x_1}{2}\right)$; 
               if $\left\{P_1,P_2\right\}$ is the diagonal of the square;
    \item[(b)] $\left(x_3,y_3\right)=\left(x_2 \pm(y_2 - y_1),y_2  \pm(x_1- x_2)\right)$, 
               $\left(x_4,y_4\right)=\left(x_1\pm(y_2 - y_1) , y_1 \pm(x_1 -x_2) \right)$ if $\left\{P_1,P_2\right\}$ is an edge of the square.               
  \end{enumerate}  
\end{lemma}

A graphical representation of the, up to similarity, unique point set with three squares through a common pair of points, without any further points, is given depicted 
in Figure~\ref{fig_3_squares_through_line_configuration}. 

\begin{figure}[htp]
\begin{center}
\begin{tikzpicture}[scale=0.8]
  \draw (2,0) node[anchor=north]{$1$};
  \draw (2,2) node[anchor=south]{$2$};
  \draw (0,0) node[anchor=north east]{$3$};
  \draw (0,2) node[anchor=south east]{$4$};
  \draw (4,0) node[anchor=north west]{$5$};
  \draw (4,2) node[anchor=south west]{$6$};
  \draw (1,1) node[anchor=east]{$7$};
  \draw (3,1) node[anchor=west]{$8$};
  \fill[black] (0,0)circle(2.5pt);
  \fill[black] (0,2)circle(2.5pt);
  \fill[black] (2,0)circle(2.5pt);
  \fill[black] (2,2)circle(2.5pt);
  \fill[black] (4,0)circle(2.5pt);
  \fill[black] (4,2)circle(2.5pt);
  \draw[thick](0,0)--(4,0);
  \draw[thick](0,2)--(4,2);
  \draw[thick](0,0)--(0,2);  
  \draw[thick](2,0)--(2,2);
  \draw[thick](4,0)--(4,2);
  \fill[black] (1,1)circle(2.5pt);
  \fill[black] (3,1)circle(2.5pt);
  \draw[thick](1,1)--(2,0);
  \draw[thick](1,1)--(2,2);
  \draw[thick](3,1)--(2,0);
  \draw[thick](3,1)--(2,2);
\end{tikzpicture}
\caption{Three squares through a pair of points -- point set $\cP_{8,3}^1$.}
\label{fig_3_squares_through_line_configuration}
\end{center}
\end{figure}
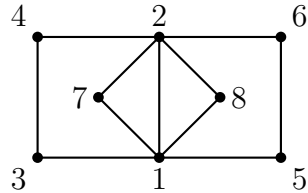
 
Since each square consists of ${4\choose 2}=6$ pairs of points, Lemma~\ref{lemma_two_points_determine_coordinates} directly 
implies the upper bound $S_{\square}(n)\le {n\choose 2} /2=\tfrac{n^2}{4}+O(n)$. Another easy implication is that through each 
triple of points there exists at most one square. So, given an $n$-point set $\cP$ we may consider every triple of points and 
compute a candidate for an $(n+1)$th point by considering a square through the new point. We call this procedure $1$-point 
extension or just $1$-extension for short.

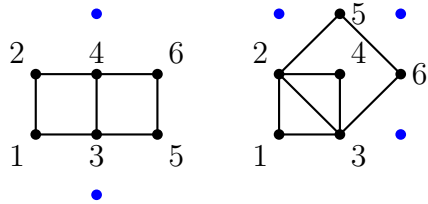
\begin{figure}[htp]
\begin{center}
\begin{tikzpicture}[scale=0.8]
  \fill[blue] (1,-1)circle(2.5pt);
  \fill[blue] (1,2)circle(2.5pt);
  \draw (0,0) node[anchor=north east]{$1$};
  \draw (0,1) node[anchor=south east]{$2$};
  \draw (1,0) node[anchor=north]{$3$};
  \draw (1,1) node[anchor=south]{$4$};
  \draw (2,0) node[anchor=north west]{$5$};
  \draw (2,1) node[anchor=south west]{$6$};
  \fill[black] (0,0)circle(2.5pt);
  \fill[black] (0,1)circle(2.5pt);
  \fill[black] (1,0)circle(2.5pt);
  \fill[black] (1,1)circle(2.5pt);
  \fill[black] (2,0)circle(2.5pt);
  \fill[black] (2,1)circle(2.5pt);
  \draw[thick](0,0)--(2,0);
  \draw[thick](0,1)--(2,1);
  \draw[thick](0,0)--(0,1);  
  \draw[thick](1,0)--(1,1);
  \draw[thick](2,0)--(2,1);
  \draw (4,0) node[anchor=north east]{$1$};
  \draw (4,1) node[anchor=south east]{$2$};
  \draw (5,0) node[anchor=north west]{$3$};
  \draw (5,1) node[anchor=south west]{$4$};
  \draw (5,2) node[anchor=west]{$5$}; 
  \draw (6,1) node[anchor=west]{$6$};
  \fill[black] (4,0)circle(2.5pt);
  \fill[black] (4,1)circle(2.5pt);
  \fill[black] (5,0)circle(2.5pt);
  \fill[black] (5,1)circle(2.5pt);
  \fill[black] (5,2)circle(2.5pt);
  \fill[black] (6,1)circle(2.5pt);
  \fill[blue] (6,0)circle(2.5pt); 
  \fill[blue] (6,2)circle(2.5pt);
  \fill[blue] (4,2)circle(2.5pt);
  \draw[thick](4,0)--(5,0);
  \draw[thick](4,0)--(4,1);
  \draw[thick](4,1)--(5,1);  
  \draw[thick](5,0)--(5,1);
  \draw[thick](4,1)--(5,0);
  \draw[thick](5,2)--(6,1);
  \draw[thick](4,1)--(5,2);
  \draw[thick](5,0)--(6,1);
\end{tikzpicture}
\caption{$1$-extension for the point sets $\cP_{6,2}^1$ and $\cP_{6,2}^2$.}
\label{fig_1_extension}
\end{center}
\end{figure}

In Figure~\ref{fig_1_extension} we have depicted the candidate points in the $1$-extension for $\cP_{6,2}^1$ and $\cP_{6,2}^2$ by blue circles. 
Since one of the three $2$-subsets of the triple of points has to be an edge of the square, we conclude from Lemma~\ref{lemma_two_points_determine_coordinates}.(b) 
that the {\lq\lq}new{\rq\rq} point lies on the integer grid if $\cP\subset \Z^2$. Of course it may happen that three points determine a fourth point that is already 
contained in the point set. Up to symmetry, i.e., similarity, the resulting point sets are depicted in Figure~\ref{fig_7_configurations}. Note that 
$\cP_{6,2}^1$ yields $\cP_{7,3}^1$ only, while $\cP_{6,2}^1$ produces both $7$-point sets. 

\begin{figure}[htp]
\begin{center}
\begin{tikzpicture}[scale=0.8]
  \fill[black] (0,0)circle(2.5pt);
  \fill[black] (0,1)circle(2.5pt);
  \fill[black] (1,0)circle(2.5pt);
  \fill[black] (2,0)circle(2.5pt);
  \fill[black] (1,1)circle(2.5pt);
  \fill[black] (1,2)circle(2.5pt);
  \fill[black] (2,1)circle(2.5pt);
  \draw[thick](0,0)--(2,0);
  \draw[thick](0,0)--(0,1);
  \draw[thick](0,1)--(2,1);  
  \draw[thick](1,0)--(2,1);
  \draw[thick](0,1)--(1,0);
  \draw[thick](1,2)--(2,1);
  \draw[thick](0,1)--(1,2);
  \draw[thick](1,0)--(1,1);
  \draw[thick](2,0)--(2,1);
\end{tikzpicture}
\quad
\begin{tikzpicture}
  \fill[black] (0,0)circle(2.5pt);
  \fill[black] (0,1)circle(2.5pt);
  \fill[black] (1,0)circle(2.5pt);
  \fill[black] (2,2)circle(2.5pt);
  \fill[black] (1,1)circle(2.5pt);
  \fill[black] (1,2)circle(2.5pt);
  \fill[black] (2,1)circle(2.5pt);
  \draw[thick](0,0)--(1,0);
  \draw[thick](0,0)--(0,1);
  \draw[thick](0,1)--(2,1);  
  \draw[thick](1,0)--(2,1);
  \draw[thick](0,1)--(1,0);
  \draw[thick](1,2)--(2,1);
  \draw[thick](0,1)--(1,2);
  \draw[thick](1,0)--(1,2);
  \draw[thick](2,1)--(2,2);
  \draw[thick](1,2)--(2,2);
\end{tikzpicture}
\caption{Two non-similar point sets $\cP_{7,3}^1$ and $\cP_{7,3}^2$.}
\label{fig_7_configurations}
\end{center}
\end{figure}
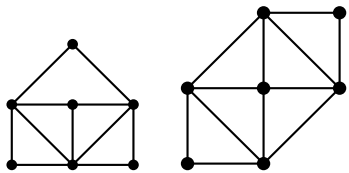

Lemma~\ref{lemma_two_points_determine_coordinates} also suggests a $2$-extension procedure, i.e., for an arbitrary pair of points of an $n$-point 
set $\cP$ consider the three possible squares containing them. Here the candidates come in pairs of {\lq\lq}new{\rq\rq} points. We depict those 
candidates by a blue line and also color the corresponding vertices blue, see Figure~\ref{fig_2_extension} for an example of $2$-extension 
applied to the unit square $\cP_{4,1}^1$. The four horizontal or vertical blue lines yield $\cP_{6,2}^1$ and the four skewed blue lines yield $\cP_{6,2}^2$. 

\begin{figure}[htp]
\begin{center}
\begin{tikzpicture}[scale=0.7]
  \fill[blue] (-1,0)circle(2.5pt);
  \fill[blue] (-1,1)circle(2.5pt);
  \fill[blue] (2,0)circle(2.5pt);
  \fill[blue] (2,1)circle(2.5pt);
  \fill[blue] (0,2)circle(2.5pt);
  \fill[blue] (1,2)circle(2.5pt);
  \fill[blue] (0,-1)circle(2.5pt);
  \fill[blue] (1,-1)circle(2.5pt);
  \draw (0,0) node[anchor=north east]{$1$};
  \draw (0,1) node[anchor=south east]{$2$};
  \draw (1,0) node[anchor=north west]{$3$};
  \draw (1,1) node[anchor=south west]{$4$};
  \fill[black] (0,0)circle(2.5pt);
  \fill[black] (0,1)circle(2.5pt);
  \fill[black] (1,0)circle(2.5pt);
  \fill[black] (1,1)circle(2.5pt);
  \draw[thick](0,0)--(1,0);
  \draw[thick](0,1)--(1,1);
  \draw[thick](0,0)--(0,1);  
  \draw[thick](1,0)--(1,1);
  \draw[thick,blue](-1,0)--(-1,1)--(0,2)--(1,2)--(2,1)--(2,0)--(1,-1)--(0,-1)--(-1,0);
\end{tikzpicture}
\caption{$2$-extension applied to $\cP_{4,1}^1$.}
\label{fig_2_extension}
\end{center}
\end{figure}
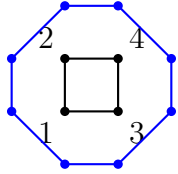

To simplify the notation, we will consider $1$-extensions as a special case of $2$-extensions, i.e., we allow that from the two {\lq\lq}new{\rq\rq} points 
one (or both, to include the degenerate case) can be already contained in the point set. Note that this case does not occur in our example in 
Figure~\ref{fig_2_extension} since there exists no $1$-extension of the unit square $\cP_{4,1}^1$ (that increases the number of points). From 
Lemma~\ref{lemma_two_points_determine_coordinates}.(a) we see that it may be possible to scale the resulting point set by a factor of two in order 
to stay within the integer grid. Note that in our example no scaling was necessary. Applying $2$-extension recursively, starting from the unit square 
$\cP_{4,1}^1$, gives us quite some non-similar $n$-point sets with $m$ squares, see Table~\ref{table_iso_types_exhaustive_2_extension}.\footnote{There are e.g.\ 
$1\,180\,723\,093$ non-similar $17$-point sets. For a 
publicly available implementation by Hugo van der Sanden see \url{https://github.com/hvds/seq/tree/master/A051602}.}   

\begin{table}[htp!]
 %% \begin{center}
 %%    \begin{tabular}{|r|r|r|r|r|r|r|r|r|r|r|r|r|r|r|r|r|r|r|}
 %%    \hline
 %%    $n/m$ & 1 & 2 & 3 & 4 & 5 & 6 & 7 & 8 & 9 & 10 & 11 \\
 %%    \hline  
 %%    4 & 1 &  &  &  &  &  &  &  &  &  &  \\
 %%    5 & &  &  &  &  &  &  &  &  &  &  \\
 %%    6 & & 2 &  &  &  &  &  &  &  &  &  \\
 %%    7 & &  & 2 &  &  &  &  &  &  &  &  \\
 %%    8 & &  & 15 & 2 &  &  &  &  &  &  &  \\
 %%    9 & &  &  & 34 & 1 & 1 &  &  &  &  &  \\
 %%    10 & &  &  & 340 & 74 & 5 & 1 &  &  &  &  \\
 %%    11 & &  &  &  & 1405 & 159 & 15 & 5 &  &  &  \\
 %%    12 & &  &  &  & 15621 & 4729 & 476 & 80 & 11 & 3 & 1 \\
 %%    \hline
 %%    \end{tabular}
 %%  \end{center}
 %%  
 %%  \bigskip  
 %%  \bigskip  
 %%  
  \begin{center}
    \begin{tabular}{llllllllllllllllll}
      \hline
      $n$ & 4 & 6 & 7 &  8 & 8 &  9 & 9 & 9 &  10 & 10 & 10 & 10 & 11   & 11  & 11 & 11 & 12    \\ 
      $m$ & 1 & 2 & 3 &  3 & 4 &  4 & 5 & 6 &   4 &  5 & 6  & 7  & 5    & 6   & 7  & 8  & 5     \\ 
      \# &  1 & 2 & 2 & 15 & 2 & 34 & 1 & 1 & 340 & 74 & 5  & 1  & 1405 & 159 & 15 & 5  & 15621 \\
      \hline
    \end{tabular}  
    \begin{tabular}{llllllllllllllll}
      \hline
      $n$ & 12   & 12  & 12 & 12 & 12 & 12 & 13    & 13    & 13   & 13  & 13 & 13 & 13 & 13 \\ 
      $m$ & 6    & 7   & 8 & 9  & 10 & 11 & 6     & 7     & 8    & 9   & 10 & 11 & 12 & 13 \\ 
      \#  & 4729 & 476 & 80 & 11 & 3  & 1  & 90573 & 15955 & 1836 & 482 & 43 & 14 & 1  & 1 \\ 
      \hline
    \end{tabular}
    \begin{tabular}{llllllllllllllll}
      \hline
      $n$ & 14      & 14     & 14    & 14   & 14   & 14  & 14 & 14 & 14 & 14 & 15      \\
      $m$ & 6       & 7      & 8     & 9    & 10   & 11  & 12 & 13 & 14 & 15 & 7       \\ 
      \#  & 1088332 & 403295 & 61386 & 9319 & 2301 & 356 & 83 & 10 & 4  & 2  & 8143021 \\ 
      \hline
    \end{tabular}
    \begin{tabular}{llllllllllllllll}
      \hline
      $n$ & 15      & 15     & 15    & 15    & 15   & 15  & 15  & 15 & 15 & 15 \\
      $m$ & 8       & 9      & 10    & 11    & 12   & 13  & 14  & 15 & 16 & 17 \\ 
      \#  & 1745837 & 273037 & 60632 & 10982 & 2693 & 460 & 122 & 26 & 7  & 2  \\
      \hline
    \end{tabular}  
    \begin{tabular}{llllllllllllllll}
      \hline
      $n$ & 16        & 16       & 16      & 16      & 16     & 16    & 16    & 16   \\
      $m$ & 7         & 8        & 9       & 10      & 11     & 12    & 13    & 14   \\ 
      \#  & 101999759 & 44513294 & 8155822 & 1445326 & 360147 & 69230 & 19076 & 3488 \\
      \hline
    \end{tabular}
    \begin{tabular}{llllllllllllllll}
      \hline
      $n$ & 16   & 16  & 16 & 16 & 16 & 16 & 17        & 17        & 17       & 17      \\
      $m$ & 15   & 16  & 17 & 18 & 19 & 20 & 8         & 9         & 10       & 11      \\ 
      \#  & 1017 & 239 & 55 & 17 & 3  & 2  & 919429357 & 215082508 & 37029433 & 7414942 \\ 
      \hline
    \end{tabular}
    \begin{tabular}{llllllllllllllll}
      \hline
      $n$ & 17      & 17     & 17    & 17    & 17   & 17  & 17 & 17 & 17 & 17 & 17 \\
      $m$ & 12      & 13     & 14    & 15    & 16   & 17  & 18 & 19 & 20 & 21 & 22 \\ 
      \#  & 1419401 & 281512 & 52643 & 10546 & 2137 & 511 & 89 & 11 & 2  & 0  & 1  \\
      \hline
    \end{tabular}
    \caption{Number of non-similar points sets $\cP$ produced by recursive $2$-extension starting from a unit square per number of points $n$ and
    squares $m=S_{\square}(\cP)$.}
    \label{table_iso_types_exhaustive_2_extension}
  \end{center}
\end{table}
%% test/extension.cpp : n -> n+1, n+2
%%
%% 10: 0m2,729s
%% 11: 0m4,973s
%% 12: 7m44,966s
%% 13: 21m14,055s
%% 14: real	1436m13,254s user	1493m6,146s sys	56m35,010s 
%% 14 (extend to 15 only): user	75m11,990s sys	4m49,905s
%% 15 (extend to 16 only): real	937m4,640s   user 982m53,839s  sys	61m20,741s

Note that if an $n$-point set $\cP$ is obtained from the recursive application of $2$-extension, starting from the unit square, then we have
$S_{\square}(\cP)\ge \left\lceil n/2\right\rceil-1$ for $n\ge 6$. In particular, no $7$-point set $\cP$ with $S_{\square}(\cP)=2$ squares will be 
obtained. Note that the configuration in Figure~\ref{fig_flexible_7_2_configuration} is non-rigid in the sense that we can twist the two squares, 
connected via vertex~$4$, without changing the number of squares in the point set. While the resulting point sets will be non-similar, we will 
discuss such transformations later on. Despite those complications, we nevertheless expect that for each $n$-point set $\cP$ with a sufficiently 
large number $S_{\square}(\cP)$ of squares there exists a similar point set that can be obtained by 
the recursive application of $2$-extension starting from $\cP_{4,1}^1$. More concretely, we state:
\begin{conjecture}
  \label{conjecture_2_extension}  
  For $n\ge 6$ every $n$-point set $\cP$ with $S_{\square}(\cP)=S_{\square}(n)$ is similar to an $n$-point set $\cP'$ obtained by 
  recursive $2$-extension starting from $\cP_{4,1}^1$.
\end{conjecture} 

\begin{figure}[htp]
\begin{center}
\begin{tikzpicture}[scale=0.5]
  \draw (0,0) node[anchor=north east]{$1$};
  \draw (0,1) node[anchor=south east]{$2$};
  \draw (1,0) node[anchor=north west]{$4$};
  \draw (1,1) node[anchor=south west]{$3$};
  \draw (0,2) node[anchor=south east]{$5$};
  \draw (2,3) node[anchor=south west]{$6$};
  \draw (3,1) node[anchor=north west]{$7$};
  \fill[black] (0,0)circle(2.5pt);
  \fill[black] (0,1)circle(2.5pt);
  \fill[black] (1,0)circle(2.5pt);
  \fill[black] (1,1)circle(2.5pt);
  \fill[black] (3,1)circle(2.5pt);
  \fill[black] (0,2)circle(2.5pt);
  \fill[black] (2,3)circle(2.5pt);
  \draw[thick](0,0)--(1,0);
  \draw[thick](0,1)--(1,1);
  \draw[thick](0,0)--(0,1);  
  \draw[thick](1,0)--(1,1);
  \draw[thick](1,0)--(0,2)--(2,3)--(3,1)--(1,0);
\end{tikzpicture}
\caption{A non-rigid $7$-point set $\cP_{7,2}^1$ with two squares.}
\label{fig_flexible_7_2_configuration}
\end{center}
\end{figure}
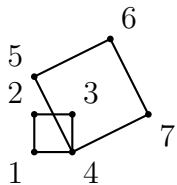

Associating $\R^2$ with $\C$ we can state compact criteria for four points forming a square.
\begin{lemma}$\,$\\[-3mm]
  \begin{itemize}
    \item[(1)] Four vertices $z_1,z_2,z_3,z_4\in\C$ of a quadrilateral in counterclockwise order form a parallelogram iff 
               $z_1-z_2+z_3-z_4=0$.\\[-7mm]
    \item[(2)] Three vertices $z_1,z_2,z_3\in \C$ form a right angle in counterclockwise order\footnote{The vector $z_3-z_1$ rotated by 
               a right angle in counterclockwise order gives the vector $z_1-z_2$.} iff $\left(z_3-z_2\right)i=z_1-z_2$.\\[-7mm]
    \item[(3)] Four vertices $z_1,z_2,z_3,z_4\in\C$ of a quadrilateral in counterclockwise order form a square iff 
               $z_1-z_2+z_3-z_4=0$ and $z_4-z_1=i\left(z_2-z_1\right)$.                        
  \end{itemize}
\end{lemma}

Our next aim are combinatorial relaxations for point sets $\cP\subset \R^2$, i.e., we want to consider discrete representations that do not list 
the coordinates of the points. An \emph{incidence structure} is a triple  $(P, L, I)$ where $P,L$ are sets and $I\subseteq P\times L$ is the \emph{incidence relation}. 

\begin{definition}
  \label{definition_oriented_square_set}
  An \emph{oriented square set} is an incidence structure $(P,S,I)$, where $P=\{1,\dots,n\}$ for some integer $n$, $S$ consists of 
  objects of the form $\left(v_1,v_2,v_3,v_4\right)$, where $1\le v_1,v_2,v_3,v_4\le n$ are pairwise different with $v_1=\min\{v_1,v_2,v_3,v_4\}$, and 
  $\left(p,\left(v_1,v_2,v_3,v_4\right)\right)\in I$ iff $p\in\left\{v_1,v_2,v_3,v_4\right\}$. We call the elements of $S$ \emph{squares}.
\end{definition}
We abbreviate an oriented square set $\cI=(P,S,I)$ by $(P,S)$ (since $I$ is well defined by $P$ and $S$), call $\#P$ the \emph{order} of $\cI$ and 
$\# S$ the cardinality $\#\cI$ of $\cI$. We also speak of an oriented square set $S$ of order $n$ referring to the oriented 
square set $\left(\left\{1,\dots,n\right\},S\right)$. 

Two oriented square sets $(P,S,I)$ and $(P',S',I')$ are called \emph{isomorphic} if there exist bijections $\alpha\colon P\to P'$ and $\beta\colon S\to S'$ 
such that $(p,s)\in I$ iff $(\alpha(p),\beta(s))\in I'$. 

\begin{definition}
  \label{definition_oriented_square_set_realizable}
  An oriented square set $(P,S)$ of order $n$ is called \emph{realizable} if their exist pairwise disjoint points $P_1,\dots,P_n\in\R^2$ such that 
  the points $P_a,P_b,P_c,P_d$ form a square in counterclockwise ordering for each $\left(a,b,c,d\right)\in S$. 
\end{definition} 

Of course the determination to the counterclockwise ordering is arbitrary and we may also require clockwise orderings 
for all squares. Given an $n$-point set $\cP\subset\R^2$ and an arbitrary labeling of the points with labels in $\{1,\dots,n\}$, 
there exists a unique oriented square set $(P,S)$ with $\#S=S_{\square}(\cP)$ that is realized by $\cP$. We call $(P,S)$ the \emph{maximal 
oriented square set} of $\cP$ and also use the notation $S(\cP)$. For any subset $S'\subseteq S$ the pair $(P,S')$ is also an oriented square 
set realized by $\cP$ and we speak of \emph{an oriented square set of $\cP$}.  

\begin{example}
  \label{example_oss_6_2_1}
  Consider the point set $\cP_{6,2}^1$ with labels and coordinates as in Figure~\ref{fig_6_configurations}. The maximal oriented square 
  set $S$ (of order $6$) of $\cP$ is given by $S=\big\{(1,3,4,2),(3,5,6,4)\big\}$. Any relabeling of the points yields an isomorphic oriented 
  square system that is also realizable. 
\end{example}

Next we want to consider the inverse problem, i.e., when is a given oriented square set $(P,S)$ realizable?
\begin{definition}
  \label{definition_equation_system}
  Let $S$ be an oriented square set of order $n$. For $n$ complex-valued variables $z_j$, where $1\le j\le n$, the linear equation system $L(n,S)$ 
  consists of the $2\cdot \#S$ equations 
  $$ 
    z_a-z_b+z_c-z_d=0\quad \text{ and }\quad z_d-z_a=i\left(z_b-z_a\right)
  $$
  for each $s=(a,b,c,d)\in S$. 
\end{definition}
\begin{lemma}
  \label{lemma_criterium_realizable}
  Let $S$ be an oriented square set of order $n$. It is realizable iff the linear equation system $L(n,S)$ admits 
  a solution satisfying $z_j\neq z_k$ for all $1\le j<k\le n$.   
\end{lemma}

\begin{example}
  \label{example_attached_equation_system_7_2}
  Consider the point set $\cP_{7,2}^1$ with coordinates and labels as in Figure~\ref{fig_flexible_7_2_configuration}. The maximal oriented square 
  set of order $7$ of $\cP_{7,2}^1$ is given by $S=\left\{(1,4,3,2),(4,7,6,5)\right\}$. Using Lemma~\ref{lemma_criterium_realizable} we can compute 
  the full space of realizations. W.l.o.g.\ we assume $z_1=0$ and $z_2=i$ for the coordinates of the points with labels $1$ and $2$, respectively.
  
  Over the reals, the solution space of $L(7,S)$ is two-dimensional and can e.g.\ be parameterized as     
  \begin{eqnarray*}
    \cP(u,v) &=& \big\{0,i,1+i,1, 1+u+vi, 1+u+v+(v-u)i,1+v,-ui\big\}\\
             &\cong&\big\{(0,0), (0,1), (1,1), (1,0), (1+u,v), (1+u+v,v-u),(1+v,-u)\big\}\subset\R^2.
  \end{eqnarray*}
  The condition $z_j\neq z_k$ for all $1\le j<k\le 7$ are equivalent to
  $$
    (u,v)\notin\big\{(-1,0),(-1,1),(0,1),(0,0),(-\tfrac{1}{2},-\tfrac{1}{2}),(-\tfrac{1}{2},\tfrac{1}{2}),(-1,-1),(0,-1)\big\}.
  $$  
  For $(u,v)=(0,1)$ the point set $\cP(u,v)$ is similar to $\cP_{6,2}^1$, for $(u,v)=(-1,1)$ we end up with $\cP_{6,2}^2$. An extreme case occurs for $(u,v)=(0,0)$, 
  where the vertices $4$, $5$, $6$, and $7$ are pairwise identical, i.e., we end up with $\cP_{4,1}$. The coordinates used in Figure~\ref{fig_flexible_7_2_configuration} 
  correspond to $(u, v) = (-1, 2)$.
\end{example}

\begin{remark}
  The point sets $\cP(u,v)$ considered in Example~\ref{example_attached_equation_system_7_2} are pairwise non-similar for different values of 
  $(u,v)\in\R^2$. This property does not hold in general and is due to our chosen specific parameterization.  If $S$ is the maximal oriented square 
  set for $\cP_{6,2}^2$, then all solutions with pairwise different coordinates $z_j$ correspond to similar point sets. Note that even a solution 
  with pairwise different coordinates may correspond to a realization $\cP'$ of $S$ with $S_{\square}(\cP')>\#S$. 
\end{remark}

Solutions of $L(n,S)$ such that there exist indices $j\neq k$ with $z_j=z_k$ are called \emph{degenerate}. We will now study the question when $L(n,S)$ 
admits a non-degenerate solution. For convenience, we will be working over $\R$ again and apply a linear transformation so 
that the solution space of $L(n,S)$ is spanned by $\lambda_1,\dots,\lambda_l\in\R$. The conditions $z_j\neq z_k$ transfer to 
conditions of the form $\sum_{j=1}^l a_j\lambda_i\neq 0$ linked as $\vee$-pairs, where the $a_j$ are rational numbers.    
\begin{lemma}
  \label{lemma_neq_zero_constraints}
  Let $a_i^j\in \Q$ for $1\le i\le l$, $1\le j\le n$, where $l$, $n$ are arbitrary integers, and $F_1,\dots, F_f$ be arbitrary 
  subsets of $\{1,\dots, n\}$. Then there exists a vector $x\in\R^l$ such that for each $1\le h\le f$ there exists an index $j\in 
  F_h$ such that $\sum_{i=1}^l a_i^j x_i\neq 0$ iff for each $1\le h\le f$ there exists an index $j\in F_h$ such that $\left(a_1^j,\dots,a_l^j\right) 
  \neq \mathbf{0}$.   
\end{lemma}
\begin{proof}
  If there exists an index $1\le h\le f$ with $\left(a_1^j,\dots,a_l^j\right)=\mathbf{0}$ for all $j\in F_h$, then no $x\in\R^l$ can 
  satisfy $\sum_{i=1}^l a_i^j x_i\neq 0$ for an index $j\in F_h$. Otherwise, let $j_h\in F_h$ denote an index with $\left(a_1^{j_h},\dots,a_l^{j_h}\right)\neq\mathbf{0}$ 
  for each $1\le h\le f$. Now choose the $x_i$ as $l$ $\Q$-linearly independent numbers. Then we have $\sum_{i=1}^l a_i^{j_h} x_i\neq 0$, since 
  $\{0\}\cup \left\{ a_i^{j_h} \,:\, 1\le i\le l\right\}\subset\Q$, for all $1\le h\le f$. 
\end{proof}
\begin{example}
  \label{example_hamming_weight_6_sets}
  The oriented square set of order $9$ given by $S=\big\{(1,3,4,2),(2,5,9,8),(4,7,6,5)\big\}$ is realizable, see Figure~\ref{figure_p_9_3_1}. Over the reals, the 
  solution space is six-dimensional. Note that the realization depicted in Figure~\ref{figure_p_9_3_1} is on the integer grid. However, it is also possible 
  to choose the side lengths of the three squares as $1$, $e$, and $\pi$, so that no similar point set $\cP'\subset\Q^2$ exists.
  
  Let us add $(3,6,8,10)$ to $S$, i.e., we consider the oriented square set of order $10$ given by $S'=\big\{(1,3,4,2),(2,5,9,8),(4,7,6,5),(3,6,8,10)\big\}$. 
  Over the reals, the solution space of $L(10,S')$ is $4$-dimensional. 
  An integer realization is given by
  $$
    \cP\big\{(0,5),(5,5),(0,0),(5,0),(7,1),(8,-1),(6,-2),(9,7),(11,3),(1,8) \big\}\subset\Z^2.
  $$ 
    
  If we further add $(1,10,7,9)$ to $S'$ and denote the resulting oriented square set by $S''$, then $L(10,S'')$ admits a $2$-dimensional solution space over the reals.  
  From Lemma~\ref{lemma_neq_zero_constraints} we conclude that $S''$ is non-realizable. 
\end{example}

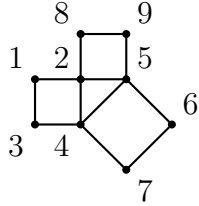
\begin{figure}[htp!]
\begin{center}
\begin{tikzpicture}[scale=0.6]
  \draw (0,0) node[anchor=north east]{$3$};
  \draw (0,1) node[anchor=south east]{$1$};
  \draw (1,0) node[anchor=north east]{$4$};
  \draw (1,1) node[anchor=south east]{$2$};
  \draw (1,2) node[anchor=south east]{$8$};
  \draw (2,2) node[anchor=south west]{$9$};
  \draw (2,1) node[anchor=south west]{$5$};
  \draw (3,0) node[anchor=south west]{$6$};
  \draw (2,-1) node[anchor=north west]{$7$};
  \fill[black] (0,0)circle(2.5pt);
  \fill[black] (0,1)circle(2.5pt);
  \fill[black] (1,0)circle(2.5pt);
  \fill[black] (1,1)circle(2.5pt);
  \fill[black] (1,2)circle(2.5pt);
  \fill[black] (2,2)circle(2.5pt);
  \fill[black] (2,1)circle(2.5pt);
  \fill[black] (3,0)circle(2.5pt);
  \fill[black] (2,-1)circle(2.5pt);
  \draw[thick](0,0)--(1,0)--(1,1)--(0,1)--(0,0);
  \draw[thick](1,1)--(1,2)--(2,2)--(2,1)--(1,1);
  \draw[thick](1,0)--(2,1)--(3,0)--(2,-1)--(1,0);
\end{tikzpicture}
\caption{A non-rigid $9$-point set $\cP_{9,3}^1$ with three squares.}
\label{figure_p_9_3_1}
\end{center}
\end{figure}

Our next goal is to show that for each $n$-point set $\cP\subset\R^2$ there exists an $n$-point set $\cP'\subset\Q^2$ with $S_{\square}(\cP')\ge S_{\square}(\cP)$. 
Let $S$ denote a maximal oriented square set of $\cP$. By assumption, $L(n,S)$ admits a solution $\left(z_1,\dots,z_n\right)\in\C^n$ with $z_j\neq z_k$ for all 
$1\le j<k\le n$. Since the space of rational solutions is dense in the space of complex solutions and $z_j\neq z_k$ are open conditions, there also exists 
a solution $\left(z_1',\dots,z_n'\right)\in (\Q[i])^n$ with $z_j'\neq z_k'$ for all $1\le j<k\le n$. Just for illustration, we give an explicit construction:    
\begin{lemma}
  \label{lemma_rational_solution_satisfying_inequalities}
  Let $A\in \Q^{m\times n}$, $B\in \Q^{l\times n}$, and $x\in \R^n$ satisfying $Ax=0$ and $Bx\neq 0$. For each $\varepsilon>0$ there exists a vector $\tilde{x}\in \Q^n$ 
  such that $A\tilde{x}=0$, $B\tilde{x}\neq 0$, and $\Vert x-\tilde{x}\Vert_\infty<\varepsilon$. 
\end{lemma}
\begin{proof}
  Choose a basis $S\subseteq \{1,\dots,n\}$ of $\left\{x\in\R^n \,:\, Ax=0\right\}$, i.e., choose $\lambda_i^j$ for all $j\in S$ and all $i\in \bar{S}:=\{1\le i\le n\,:\, i\notin S\}$
  such that
  $$
    \left\{x\in\R^n\,:\, Ax=0\right\}=\left\{x\in R^n\,:\, x_j=\sum_{i\in  \bar{S}} \lambda_i^j x_i\quad \forall j\in S\right\}.
  $$  
  Note that $A\in\Q^{m\times n}$ implies $\lambda_i^j\in\Q$. Set $\Lambda:=\max\left\{ \left|\lambda_i^j\right|\,:\, i\in\bar{S}, j\in S\right\}$, $\mu:=\Vert Bx\Vert_\infty$, 
  and $\beta:=\max\left\{ \left|b_{i,j}\right|\,:\, 1\le i\le n, 1\le j\le l\right\}$, where $B=(b_{i,j})$. W.l.o.g.\ we assume that $\beta n\varepsilon<\mu$. Set $\varepsilon'
  :=\min\!\left\{\varepsilon,\frac{\varepsilon}{n\Lambda}\right\}$ and choose 
  $\tilde{x}_i\in \Q$ such that $\vert x_i-\tilde{x}_i\vert<\varepsilon'\le \varepsilon$ for all $i\in\bar{S}$. For $j\in S$ we set $\tilde{x}_j=\sum_{i\in  \bar{S}} \lambda_i^j \tilde{x}_i$, 
  so that 
  $$
    \left\vert x_j-\tilde{x}_j\right\vert=\big\vert\sum_{i\in\bar{S}} \lambda_i^j \cdot\left(x_i-\tilde{x}_i\right)\big\vert \le 
    \sum_{i\in\bar{S}} \big\vert \lambda_i^j \cdot\left(x_i-\tilde{x}_i\right) \big\vert< n\Lambda\varepsilon'\le \varepsilon, 
  $$   
  which implies $\Vert x-\tilde{x}\Vert_\infty<\varepsilon$. Next we compute
  $$
    \Vert Bx-B\tilde{x}\Vert_\infty= \max_{1\le j\le l} \left|\sum_{i=1}^n b_{i,j} \cdot \left(x_i-\tilde{x}_i\right) \right| \le \beta\cdot\sum_{i=1}^n \left| x_i-\tilde{x}_i\right| 
    <\beta n\varepsilon<\mu,  
  $$
  so that $B\tilde{x}\neq 0$.
\end{proof}

\begin{theorem}
  \label{thm_grid_assumption}
  For each $n$-point set $\cP\subset\R^2$ there exists an $n$-point set $\cP'\subset\Q^2$ with $S_{\square}(\cP')\ge S_{\square}(\cP)$.  
\end{theorem}
\begin{proof}
  Let $S$ be the maximal oriented square set of order $n$ of $\cP$. By Lemma~\ref{lemma_criterium_realizable} the existing realization 
  $\cP$ corresponds to a solution $\left(z_1,\dots,z_n\right)\in\C^n$ of $L(n,S)$ satisfying $z_j\neq z_k$ for all $1\le j<k \le n$. After a 
  suitable transformation we can apply Lemma~\ref{lemma_rational_solution_satisfying_inequalities} to construct a solution $\left(z_1',\dots,z_n'\right)\in (\Q(i))^n$  
  of $L(n,S)$ satisfying $z_j'\neq z_k'$ for all $1\le j<k \le n$. From Lemma~\ref{lemma_criterium_realizable} we then conclude that 
  $\cP':=\left\{z_1',\dots,n_n'\right\} \subset\Q(i)$ realizes $s$. Thus, we have $S_{\square}(\cP')\ge\#S=S_{\square}(\cP)$.    
\end{proof}
We remark that the proof of Theorem~\ref{thm_grid_assumption} and the parameter $\varepsilon$ in Lemma~\ref{lemma_rational_solution_satisfying_inequalities} 
would allow the stronger statement that we can assume the existence of a pairing between the points of $\cP$ and the points of $\cP'$ where the pairs of points 
have distance at most $\varepsilon$ in the $\Vert\cdot\Vert_\infty$-metric, i.e., $\cP'$ arises from $\cP$ by a sufficiently small perturbation.

\begin{corollary}
  \label{corollary_grid_assumption}
  For each $n$-point set $\cP\subset\R^2$ there exists an $n$-point set $\cP'\subset\Z^2$ with $S_{\square}(\cP')\ge S_{\square}(\cP)$.  
\end{corollary}

\begin{corollary}
  For each $n$-point set $\cP\subset\R^2$ obtained by $2$-extension there exists an $n$-point set $\cP'\subset\Z^2$ with $S_{\square}(\cP')= S_{\square}(\cP)$.
\end{corollary}

\begin{remark}
  Theorem~\ref{thm_grid_assumption} can be directly generalized to: For each $n$-point set $\cP\subset\R^2$ and each $\emptyset\neq \cQ\subset \Q^2$ there exists 
  an $n$-point set $\cP'\subset\Q^2$ with $S_{\square}(\cP')\ge S_{\square}(\cP)$. I.e., one can choose $\cQ$ as an isosceles right triangle $\rit$. 
  
  The approach, based on linear equation systems over the rationals, also works for e.g.\ rectangles or axis-parallel squares and rectangles. For an equilateral 
  triangle $\cQ$ we have a similar statement replacing the Gaussian rationals (integers) by the Eisenstein rationals (integers).
\end{remark}

\begin{proposition}
  \label{prop_grid_size_bound}
  For each $n$-point set $\cP\subset\R^2$ there exists an $n$-point set $\cP'\subset\left\{0,1,\dots,\Lambda\right\}^2$ with $S_{\square}(\cP')\ge S_{\square}(\cP)$ 
  and $\Lambda\le 25^n$.  
\end{proposition}
\begin{proof}
  Let $S$ be the maximal oriented square set of order $n$ of $\cP$. By Lemma~\ref{lemma_criterium_realizable} the existing realization 
  $\cP$ corresponds to a solution $\left(z_1,\dots,z_n\right)\in\C^n$ of $L(n,S)$ satisfying $z_j\neq z_k$ for all $1\le j<k \le n$. 
  W.l.o.g.\ we assume that real and imaginary parts of the $z_j$ are non-negative and that 
  $$
   \max\!\big\{ \left|\operatorname{Re}(z_j)-\operatorname{Re}(z_k)\right|, \left|\operatorname{Im}(z_j)-\operatorname{Im}(z_k)\right|\big\}  \ge 1
  $$  
  for all $1\le j<k\le n$. Consider the following linear program with variables $x_j,y_j\in \R_{\ge0 }$ for $1\le j\le n$. We convert the equations 
  of $L(n,S)$ into their real counterparts using the $x_j$ and $y_j$ variables ($4m$ equations). 
  If $\operatorname{Re}(z_j)-\operatorname{Re}(z_k)\ge 1$ we add the constraint $x_j-x_k\ge 1$ and if $\operatorname{Re}(z_j)-\operatorname{Re}(z_k)\le -1$ 
  we add the constraint $x_k-x_j\ge 1$, where $1\le j<k\le n$. Similarly, If $\operatorname{Im}(z_j)-\operatorname{Im}(z_k)\ge 1$ we add the constraint 
  $y_j-y_k\ge 1$ and if $\operatorname{Im}(z_j)-\operatorname{Im}(z_k)\le -1$ we add the constraint $y_k-y_j\ge 1$, where $1\le j<k\le n$. As target we choose 
  the minimization of $\sum_{j=1}^n x_j\,+\,\sum_{j=1}^n y_j$, so that the LP is bounded. The existence of the solution $\left(z_1,\dots,z_n\right)\in\C^n$ 
  implies that the LP is also feasible. Note that all coefficients of the LP formulation are contained in $\{-1,0,1\}$ and that each constraint contains at 
  most four non-zero coefficients on the left-hand side. Consider a basic solution of the LP, i.e., a solution of the uniquely solve able equation 
  system $A\cdot (x,y)^\top=b$, where $A$ is a suitable submatrix of the coefficient matrix of the LP and $b$ a suitable subvector of the corresponding right hand 
  side. For a variable $x_j$ (or $y_j$) let $A^{x_j}$ (or $A^{y_j}$) denote the matrix arising from $A$ when the column corresponding to $x_j$ (or $y_j$) is replaced by 
  $b$. With this Cramer's rule yields $x_j=\det(A^{x_j})/\det(A)$ and $y_j=\det(A^{y_j})/\det(A)$ for all $1\le j\le n$. Using the 
  Leibniz formula for determinants we conclude $|\det(A^{x_j})|,|\det(A^{y_j})|\le 5^{2n}$ for all $1\le j\le n$ (and $|\det(A)|\le 4^{2n}$). Now observe 
  that $\tilde{x}_j:=x_j\cdot \det(A)=\det(A^{x_j})$, $\tilde{y}_j:=y_j\cdot \det(A)=\det(A^{y_j})$ is also a solution with $\tilde{x}_j,\tilde{y}_j\in\N$ and 
  $\left|\tilde{x}_j\right|,\left|\tilde{y}_j\right|\le 25^n$, where $1\le j\le n$.          
\end{proof}

Having the combinatorial structure of oriented square sets at hand, we can also treat non-rigid $n$-point sets using a finite number of cases only. We have to abandon the 
idea of classifying point sets up to similarity and consider equivalence classes of point sets distinguished by different oriented square sets. This causes some complications but allows 
us to determine $S_{\square}(n)$ exactly by a finite amount of computation, depending on $n$ -- in principle.   
Note that for the unit distance problem $C_{\left\{q_1,q_2\right\}}(n)$ mentioned in the introduction, some extremal examples 
are not rigid, see e.g.\ \cite[Section 5.1]{RPiDG} or \cite{diplomarbeit_schade}. However, since the numbers in Table~\ref{table_iso_types_exhaustive_2_extension} provide a 
lower bound for the number of non-isomorphic realizable connected oriented square sets with order $n$ and cardinality $m$, exhaustive enumerations will become computationally 
infeasible if $n$ gets too large. So, in the subsequent section develop tools and criteria that allow us to obtain classification results for special 
values of $n$ and $m$ without having the full classification for all $n'<n$ and $m'<m$ at hand. 

\section{Bounds for $S_{\square}(n)$}
\label{sec_bounds_squares}
Let $\preceq$ denote the lexicographical ordering on $\R^2$ and $\cP\subset\R^2$ be an arbitrary point set. Assume that the vertices of a square $s$ in $\cP$ are 
given by $v_1\prec v_2\prec v_3\prec v_4$. With this, we denote the six pairs of vertices by $e_1:=\left\{v_1,v_2\right\}$, $e_2:=\left\{v_1,v_3\right\}$, 
$e_3:=\left\{v_2,v_4\right\}$, $e_4:=\left\{v_3,v_4\right\}$, $d_1:=\left\{v_1,v_4\right\}$, and $d_2:=\left\{v_2,v_3\right\}$. Observe that the 
$e_i$ form the edges and the $d_i$ form the diagonals of $s$. We call $e_1$ the \emph{leftmost edge} and $e_2$ the \emph{second leftmost edge} of $s$. In the 
following auxiliary result we determine the possible types of pairs of points in their three squares where they are contained, cf.~Figure~\ref{fig_lex_ordering}. 

\begin{lemma}
  Let $a=(0,0)$, $b_1=(u,v)$, and $b_2=(v,-u)$.
  \begin{enumerate}
    \item[(1)] For  $0\le u<v$ the squares through $\left\{a,b_1\right\}$ are given by vertices $a\prec b_1\prec b_2 \prec (u+v,v-u)$, 
               $(-v,u)\prec(u-v,u+v)\prec a\prec b_1$, and $\tfrac{1}{2}(u-v,u+v)\prec a\prec b_1\prec \tfrac{1}{2} (u+v,v-u)$. The corresponding types 
               of $\left\{a,b_1\right\}$ are $e_1$, $e_4$, and $d_2$, respectively.\\[-6mm] 
    \item[(2)] For  $0\le u<v$ the squares through $\left\{a,b_2\right\}$ are given by vertices $a\prec b_1\prec b_2 \prec (u+v,v-u)$, 
               $(-u,-v)\prec a \prec (v-u,-u-v)\prec b_2$, and $a\prec \tfrac{1}{2}(v-u,-u-v)\prec \tfrac{1}{2} (u+v,v-u)\prec b_2$. The corresponding types 
               of $\left\{a,b_2\right\}$ are $e_2$, $e_3$, and $d_1$, respectively..\\[-6mm]
    \item[(3)] For  $0<v\le u$ the squares through $\left\{a,b_1\right\}$ are given by vertices $a\prec b_2\prec b_1\prec (u+v,v-u)$, $(-v,u)
               \prec a\prec (u-v,u+v)\prec b_1$, and $a\prec \tfrac{1}{2}(u-v,u+v)\prec b_1\prec \tfrac{1}{2} (u+v,v-u)$. The corresponding types 
               of $\left\{a,b_1\right\}$ are $e_2$, $e_3$, and $d_1$, respectively..\\[-6mm]            
    \item[(4)] For  $0<v\le u$ the squares through $\left\{a,b_2\right\}$ are given by vertices $a\prec b_2\prec b_1 \prec (u+v,v-u)$, 
               $(-u,-v)\prec (v-u,-u-v)\prec a \prec b_2$, and $\tfrac{1}{2}(v-u,-u-v)\prec a\prec b_2\prec \tfrac{1}{2} (u+v,v-u)$. The corresponding types 
               of $\left\{a,b_2\right\}$ are $e_1$, $e_4$, and $d_2$, respectively.                        
  \end{enumerate}    
\end{lemma}

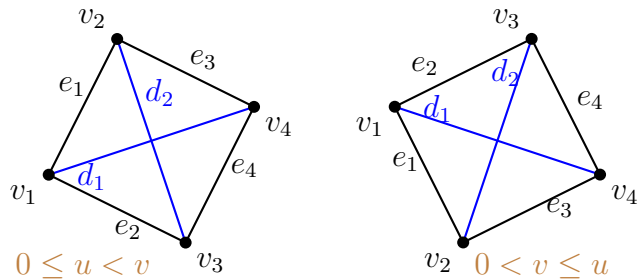
\begin{figure}[htp!]
\begin{center}
\begin{tikzpicture}[scale=0.9]
  \draw[thick](0,0)--(1,2)--(3,1)--(2,-1)--(0,0);
  \draw[thick,blue](0,0)--(3,1);
  \draw[thick,blue](1,2)--(2,-1);
  \fill[black] (0,0)circle(2.5pt);
  \fill[black] (1,2)circle(2.5pt);
  \fill[black] (3,1)circle(2.5pt);
  \fill[black] (2,-1)circle(2.5pt);
  \draw (0,0) node[black,anchor=north east]{$v_1$};
  \draw (1,2) node[black,anchor=south east]{$v_2$};
  \draw (3,1) node[black,anchor=north west]{$v_4$};
  \draw (2,-1) node[black,anchor=north west]{$v_3$};
  \draw (0.7,1) node[black,anchor=south east]{$e_1$};
  \draw (0.8,-0.51) node[black,anchor=north west]{$e_2$};
  \draw (1.9,1.4) node[black,anchor=south west]{$e_3$};
  \draw (2.5,-0.2) node[black,anchor=south west]{$e_4$};
  \draw (1,0.33) node[blue,anchor=north east]{$d_1$};
  \draw (1.25,0.85) node[blue,anchor=south west]{$d_2$};
  \draw (0.5,-1) node[brown,anchor=north]{$0\le u<v$};
\end{tikzpicture}
\quad
\begin{tikzpicture}[scale=0.9]
  \draw[thick](0,0)--(2,1)--(3,-1)--(1,-2)--(0,0);
  \draw[thick,blue](0,0)--(3,-1);
  \draw[thick,blue](2,1)--(1,-2);
  \fill[black] (0,0)circle(2.5pt);
  \fill[black] (2,1)circle(2.5pt);
  \fill[black] (3,-1)circle(2.5pt);
  \fill[black] (1,-2)circle(2.5pt);
  \draw (0,0) node[black,anchor=north east]{$v_1$};
  \draw (2,1) node[black,anchor=south east]{$v_3$};
  \draw (3,-1) node[black,anchor=north west]{$v_4$};
  \draw (1,-2) node[black,anchor=north east]{$v_2$};
  \draw (0.8,0.3) node[black,anchor=south east]{$e_2$};
  \draw (-0.2,-0.51) node[black,anchor=north west]{$e_1$};
  \draw (2.05,-1.78) node[black,anchor=south west]{$e_3$};
  \draw (2.5,-0.2) node[black,anchor=south west]{$e_4$};
  \draw (1,0.33) node[blue,anchor=north east]{$d_1$};
  \draw (1.25,0.15) node[blue,anchor=south west]{$d_2$};
  \draw (2.15,-2) node[brown,anchor=north]{$0<v \le u$};
\end{tikzpicture}
\caption{The lexicographical ordering of the vertices and the types of pairs of vertices of a square.}
\label{fig_lex_ordering}
\end{center}
\end{figure}

\begin{corollary}
  Let $\cP\subset \R^2$ be an arbitrary point set. No pair of different points in $\cP$ can form both the leftmost edge of a square and the second leftmost edge 
  of another square.
\end{corollary}

\begin{proposition}
  \label{proposition_general_upper_bound}
  For $n\ge 1$ we have $S_{\square}(n)\le\frac{n^2-1}{8}$.
\end{proposition}  
\begin{proof}
  Let $\cP\subset\R^2$ be an arbitrary $n$-point set. Using a suitable similarity transformation we can assume that at most $\tfrac{n+1}{2}$ points have 
  negative and $\tfrac{n+1}{2}$ points have positive $x$-coordinates, while the $y$-axis is free of points. By $a_{ij}$ we denote the number of squares such 
  that exactly $i$ of its vertices have negative coordinates, where $0\le i,j\le 4$ with $i+j=4$. If necessary by reflecting in the $y$-axis we can assume 
  $a_{31}+a_{40}\ge a_{13}+a_{04}$. Let $s$ be an arbitrary square of $\cP$ with vertices $v_1\prec v_2\prec v_3\prec v_4$ such that $v_1$ and $v_2$ have 
  negative $x$-coordinates. Counting pairs of points with negative $x$-coordinate that are of type $e_1$ or type $e_2$ gives  
  $$
    a_{22}+2a_{31}+2a_{40} \le {\frac{n+1}{2} \choose 2}=\frac{n^2-1}{8},
  $$
  so that
  $$
    S_{\square}(\cP)=a_{22}+a_{31}+a_{40}+a_{13}+a_{04}\le a_{22}+2a_{31}+2a_{40} \le\frac{n^2-1}{8}.  
  $$
\end{proof}
The underlying idea of the proof of Proposition~\ref{proposition_general_upper_bound} fits into the framework of a general method based on ordering relations  
as introduced in \cite{abrego2016number,kolesnikov2015number}.   

Next we want to determine the exact values of $S_{\square}(n)$ for small values of $n$. First we introduce more notation. For brevity, an $n$-point set $\cP\in\R^2$ with 
$S_{\square}(\cP)=m$ is also called \emph{$(n,m)$-configuration} in this section. We also use intuitive notations like $(n,\ge m)$-configuration. Any $(n,m)$-configuration where a pair of 
points is contained in at least two different squares contains one of the $(6,2)$-configurations $\cP_{6,2}^1$ or $\cP_{6,2}^2$ as subconfiguration, see Figure~\ref{fig_6_configurations}.  
For arbitrary point sets $\cF_1,\dots,\cF_l$ we denote by $S_{\square}(n;\cF_1,\dots,\cF_l)$ the maximum value of $S_{\square}(\cP)$ where $\cP$ is an $n$-point set such that 
no subset of its points is similar to $\cF_i$ for an index $1\le i\le l$. Besides trivial cases like $S_{\square}\!\!\left(n;\cP_{4,1}^1\right)=0$ and 
$S_{\square}\!\!\left(n;\cP_{6,2}^\star,\cP_{7,2}^1\right)=\left\lfloor n/4\right\rfloor$, the exact value of $S_{\square}(n;\cF_1,\dots,\cF_l)$ is hard to determine if $n$ is 
not too small. Nevertheless, this notation helps to better structure the subsequent arguments. 

Let $A(n,d,w)$ denote the maximum size of a binary code with word length $n$, minimum Hamming distance $d$, and constant weight $w$, see e.g.\ 
\cite{brouwer2006new} for details and bounds. A binary code of size $m$, length $n$, minimum Hamming distance $6$, and constant weight $4$ is in 
one-to-one correspondence to a set $S$ of $m$ four-subsets of $\{1,\dots,n\}$ such that $\# (a\cap b)\le 1$ for all $a,b\in S$ with $a\neq b$. 
Thus, we have 
$$
  S_{\square}\!\!\left(n;\cP_{6,2}^\star\right):=S_{\square}\!\!\left(n;\cP_{6,2}^1,\cP_{6,2}^2\right)\le A(n,6,4)
  \le \left\lfloor n\cdot \left\lfloor \frac{n-1}{3}\right\rfloor /4\right\rfloor=\frac{n^2}{12}+O(n).   
$$
The values $A(n,6,4)$ are known exactly, see \cite[Theorem 6]{brouwer2006new}, and we list the first few in Table~\ref{table_a_n_6_4}, cf.\ sequence A004037.
\begin{table}[htp!]
  \begin{center}
    \begin{tabular}{llllllllllllllllllll}
    \hline 
    $n$        & 1 & 2 & 3 & 4 & 5 & 6 & 7 & 8 & 9 & 10 & 11 & 12 & 13 & 14 & 15 & 16 & 17 \\ 
    $A(n,6,4)$ & 0 & 0 & 0 & 1 & 1 & 1 & 2 & 2 & 3 & 5  & 6  & 9  & 13 & 14 & 15 & 20 & 20 \\
    \hline 
    \end{tabular}
    \caption{The exact values of $A(n,6,4)$ for $n\le 17$.}
    \label{table_a_n_6_4}
  \end{center} 
\end{table}

Our next goal is to use the results obtained by exhaustive recursive $2$-extension starting from a unit square, see Table~\ref{table_iso_types_exhaustive_2_extension}.
\begin{definition}
  Let $\cP\subset \R^2$ be an arbitrary point set. A point set $\cP'\subseteq \cP$ is called a \emph{$2$-extension subconfiguration} if $\cP'$ can be obtained by recursive 
  $2$-extension starting from a unit square. If there is no proper superset of $\cP'$ that can also be obtained by recursive $2$-extension starting from $\cP'$, 
  then we call $\cP'$ $2$-extension maximal.
  %%\footnote{We may also use set-inclusion maximality, but it remains unclear if such a technical complication pays off.} 
\end{definition}
For an arbitrary point set $\cP\subset\R^2$ we define $S_{\rit-3\square}(\cP):=S_{\rit}(\cP)-3S_{\square}(\cP)$ and denote the corresponding maximum 
value of an $n$-point set by $S_{\rit-3\square}(n)$.
\begin{lemma}
  \label{lemma_2_extension_maximal_mixed_triangle_square}
  Let $\cP\subset\R^2$ be an arbitrary $n$-point set and $\cP'$ be a $2$-extension maximal subconfiguration. Then we have 
  $$
    S_{\square}(\cP) \le S_{\square}(\cP') + S_{\rit-3\square}(\cP\backslash\cP')\le S_{\square}(n')+S_{\rit-3\square}(n-n'), 
  $$
  where $n':=\#\cP'$.
\end{lemma}
\begin{proof}
  Let $s$ be an arbitrary square in $\cP$ that is not contained in $\cP'$. Since $\cP'$ is $2$-extension maximal either $3$ or $4$ vertices of $s$ have to be contained 
  in $\cP\backslash\cP'$. In the first case the $3$ vertices form a right isosceles triangle. In the second case the $4$ vertices form a square that 
  contains four right isosceles triangles. Note that each right isosceles triangle uniquely determines a square. 
\end{proof}

\begin{lemma}
  \label{lemma_ub_mixed_triangle_square}
  We have $S_{\rit-3\square}(n)=0$ for $n\le 2$, $S_{\rit-3\square}(3)=1$, $S_{\rit-3\square}(4)=3$, $S_{\rit-3\square}(5)=5$, $S_{\rit-3\square}(6)=7$, $S_{\rit-3\square}(7)=10$, 
  $S_{\rit-3\square}(8)=14$, and $S_{\rit-3\square}(9)=18$.
\end{lemma}
\begin{proof}
  The statements are obvious for $n\le 3$. In Appendix~\ref{appendix_dissimilar_blbk_point_sets_mixed} we present point sets attaining the mentioned values for $S_{\rit-3\square}(\cP)$. 
  Looping over all $n$-point sets with $n\le 9$ that can be obtained by $1$-extension does not yield point sets with larger values, so that the result is implied by  
  Proposition~\ref{proposition_1_extension_obtainable} and Proposition\ref{proposition_1_extension_obtainable_2}. Let $\cP'$ be a $2$-extension maximal subconfiguration of $\cP$ 
  with maximum possible cardinality. Note that we have $\#\cP'\ge 6$.
\end{proof}

\begin{theorem}
  \label{thm_exact_values_squares_1}
  We have $S_{\square}(n)=0$ for $n\le 3$, $S_{\square}(4)=S_{\square}(5)=1$, $S_{\square}(6)=2$, $S_{\square}(7)=3$, $S_{\square}(8)=4$, $S_{\square}(9)=6$, 
  $S_{\square}(10)=7$, $S_{\square}(11)=8$, $S_{\square}(12)=11$, and $S_{\square}(13)=13$. Moreover, for $n\le 12$ the extremal examples can be obtained by 
  recursive $2$-extension starting from the unit square and are listed in Appendix~\ref{appendix_dissimilar_blbk_point_sets_squares}.
\end{theorem}
\begin{proof}
  The statements are obvious for $n\le 4$. As shown by the examples in Appendix~\ref{appendix_dissimilar_blbk_point_sets_squares}, the mentioned values can indeed be attained. 
  If there are at least two squares and $S_{\square}(\cP)=S_{\square}(\# \cP)$, then the values in Table~\ref{table_a_n_6_4} imply $\#\cP=13$, so that we can assume the 
  existence of a pair of points that is contained in at least two different squares. By recursively applying Lemma~\ref{lemma_2_extension_maximal_mixed_triangle_square} 
  and Lemma~\ref{lemma_ub_mixed_triangle_square} we conclude $\#\cP'=\#\cP$. 
\end{proof}

In order to show that all $13$-point sets with $13$ squares can be obtained by recursive $2$-extension starting from the unit square it remains to show 
$S_{\square}\!\left(13,\cP_{6,2}^\star\right)<13$. Starting from the upper bound $S_{\square}\!\left(n,\cP_{6,2}^\star\right)\le A(n,6,4)$, we remark that 
$S_{\square}\!\left(n,\cP_{6,2}^\star\right)$ is the 
largest size of a binary code with word length $n$, minimum Hamming distance $6$, and constant weight $4$ that can be represented by $n$ (pairwise different) points 
in the Euclidean plane $\R^2$ such that the codewords are given by the squares spanned by the point set. In other words, we consider the maximum number 
of squares spanned by an $n$-point set such that no pair of points is contained in two different squares. 

Trivially we have  $S_{\square}\!\left(n,\cP_{6,2}^\star\right)=0$ for $n\le 3$ and $S_{\square}\!\left(n,\cP_{6,2}^\star\right)=1$ for $4\le n\le 6$. It is also easy to find 
examples showing $S_{\square}\!\left(7,\cP_{6,2}^\star\right)=S_{\square}\!\left(8,\cP_{6,2}^\star\right)=2$ and $S_{\square}\!\left(9,\cP_{6,2}^\star\right)=3$, so that the upper 
bound $S_{\square}\!\left(n,\cP_{6,2}^\star\right)= A(n,6,4)$ is attained. Examples showing showing $S_{\square}\!\left(10,\cP_{6,2}^\star\right)\ge 4$, $S_{\square}\!\left(11,\cP_{6,2}^\star\right)\ge 5$, 
$S_{\square}\!\left(12,\cP_{6,2}^\star\right)\ge 6$, and $S_{\square}\!\left(13,\cP_{6,2}^\star\right)\ge 7$ are given in Appendix~\ref{appendix_hamming}. 

\begin{example}
  Consider the (realizable) oriented square set 
  $$
    S=\big\{(1,2,3,4),(4,5,6,7),(3,8,9,7),(2,10,9,5)\big\}
  $$
  of order $10$ and cardinality $4$. It can be obtained by $1$-extension from the (realizable) oriented square set $\big\{(1,2,3,4),(4,5,6,7),(3,8,9,7)\big\}$ of order $9$ 
  and cardinality $3$. The, up to similarity, unique representation of $S$ is depicted in Figure~\ref{fig_four_imply_five}. Thus, $(3,9,4,6)$ also has to be a square, so 
  that $(3,8,9,7)$ and $(3,9,4,6)$ share two common vertices. 
\end{example}
The approach to check the realizability of $S$ via the linear equation system $L(n,S)$ with the additional inequalities $z_j\neq z_k$ can be easily extended to also 
explicitly forbid squares not contained in $S$. Transforming the formulation over $\C$ to a formulation over $\R$ we stay within the setting of Lemma~\ref{lemma_neq_zero_constraints}, 
so that we can check the existence of solutions algorithmically.

\begin{figure}[htp!]
\begin{center}
\begin{tikzpicture}[scale=0.9]
  \draw[thick](0.04,1)--(1,1)--(1,1.96)--(0.04,1.96)--(0.04,1);
  \draw[thick](1,2.04)--(2,2.04)--(2,3)--(1,3)--(1,2.04);
  \draw[thick,blue](3,2)--(2,0)--(0,1)--(1,3)--(3,2);
  \draw[thick,blue](1,3)--(2,2)--(1,1)--(0,2)--(1,3);
  \draw[thick,green](0,0)--(0,2)--(2,2)--(2,0)--(0,0);
  \fill[black] (0,0)circle(2.5pt);
  \fill[black] (2,0)circle(2.5pt);
  \fill[black] (0,1)circle(2.5pt);
  \fill[black] (0,2)circle(2.5pt);
  \fill[black] (1,1)circle(2.5pt);
  \fill[black] (1,2)circle(2.5pt);
  \fill[black] (1,3)circle(2.5pt);
  \fill[black] (2,3)circle(2.5pt);
  \fill[black] (2,2)circle(2.5pt);
  \fill[black] (3,2)circle(2.5pt);   
  \draw (0,0) node[black,anchor=north east]{$1$};
  \draw (2,0) node[black,anchor=north west]{$2$};
  \draw (0,1) node[black,anchor=north east]{$5$};
  \draw (0,2) node[black,anchor=north east]{$4$};
  \draw (1,1) node[black,anchor=north west]{$6$};
  \draw (1,2) node[black,anchor=north west]{$7$};
  \draw (1,3) node[black,anchor=south east]{$9$};
  \draw (2,2) node[black,anchor=north west]{$3$};
  \draw (2,3) node[black,anchor=south west]{$8$};
  \draw (3,2) node[black,anchor=south west]{$10$};
\end{tikzpicture}
\caption{Four squares implying a fifth square.}
\label{fig_four_imply_five}
\end{center}
\end{figure}
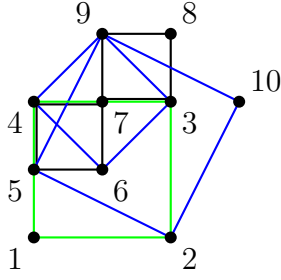

An easy averaging argument implies the existence of subconfigurations with relatively many squares:
\begin{lemma}
  \label{lemma_average_configuration} 
  Let $\cP$ be an $(n,\ge m)$-configuration. For each $1\le n'<n$ there exists an $(n',\ge l)$-subconfiguration, where $l=\left\lceil m\cdot {{n-4}\choose{n'-4}}/{n\choose{n'}}\right\rceil$.
\end{lemma}
\begin{proof}
  The average number $x$ of squares contained in an $n'$-subset of $\cP$ satisfies ${n\choose{n'}}\cdot x/{{n-4}\choose{n'-4}}=m$, 
  so that $x=m\cdot {{n-4}\choose{n'-4}}/{n\choose{n'}}$.
\end{proof}
For a given $n$-point set $\cP$ let the degree of every vertex be the number of squares in which it is contained, $\delta_{\min}$ be the minimum, and $\delta_{\max}$ 
be the maximum degree.  
\begin{lemma}
  \label{lemma_minimum_degree} 
  Let $\cP$ be an $(n,m)$-configuration and $\delta_{\min}$ be its minimum degree. Then we have $\delta_{\min}\le \left\lfloor 4m/n\right\rfloor$, 
  $\delta_{\max}\ge \left\lceil 4m/n\right\rceil$, and there exists an $(n-1,m-\delta_{\min})$-subconfiguration.  
\end{lemma}
\begin{proof}
  For the average degree $\delta$ we have $n\delta=4m$, so that $\delta=4m/n$. Removing a vertex with degree $\delta_{\min}$ from $\cP$ gives the 
  desired subconfiguration.
\end{proof}

Via exhaustive enumeration of oriented square sets we have verified $S_{\square}\!\left(10,\cP_{6,2}^\star\right)=4$, $S_{\square}\!\left(11,\cP_{6,2}^\star\right)=5$, and 
$S_{\square}\!\left(12,\cP_{6,2}^\star\right)=6$, so that Lemma~\ref{lemma_average_configuration} implies $S_{\square}\!\left(13,\cP_{6,2}^\star\right)\le 9$ and 
$S_{\square}\!\left(14,\cP_{6,2}^\star\right)\le 13$. 
\begin{conjecture}
  $$S_{\square}\!\left(13,\cP_{6,2}^\star\right)=7$$
\end{conjecture}
\begin{lemma}
  $$S_{\square}\!\left(13,\cP_{6,2}^\star\right)\le 8$$
\end{lemma}
\begin{proof}
  Let $\cP$ be a $13$-point set such that no pair of squares shares a common pair of vertices. Since the maximum degree of $\cP$ is at least $3$, 
  we assume that the squares contained in $P=\big\{ \{1,2,3,4\}, \{1,5,6,7\}, \{1,8,9,10\}\big\}$ are spanned by $\cP$. For $2\le i\le 13$ let 
  $C_i:=\left\{1,\dots,13\right\}\backslash\{i\}$ and $s_i$ be the number of squares of $\cP$ that are not contained in $P$ and whose vertices 
  are contained in $C_i$. Since $S_{\square}\!\left(12,\cP_{6,2}^\star\right)=6$, we have $s_i\le 4$ for $2\le i\le 10$ and $s_i\le 3$ for 
  $11\le i\le 13$. Since the vertices of each square are contained in at least $8$ of the $C_i$, we have
  $$
    S_{\square}(\cP)\le \#P + \left\lfloor\frac{\sum_{i=2}^{13} s_i}{8}\right\rfloor\le 3+\left\lfloor 45/8\right\rfloor=8. 
  $$ 
\end{proof}
With this, we can conclude $S_{\square}\!\left(14,\cP_{6,2}^\star\right)\le\left\lfloor 14\cdot 8/10\right\rfloor=11$ from Lemma~\ref{lemma_average_configuration}.  
%% Unfortunately, this does not allow us to improve the upper bound $S_{\square}\!\left(n,\cP_{6,2}^\star\right)\le A(n,6,4)$ for $n\in\{15,16\}$.
\begin{conjecture}
  $$S_{\square}\!\left(16,\cP_{6,2}^\star\right)<20=A(16,6,4)$$
\end{conjecture}

In the following our aim is to find criteria that guarantee a relatively large $2$-extension maximal subconfiguration of a point set $\cP$.
\begin{definition}
  \label{def_decomposition}
  Let $\cP\subset\R^2$ be an arbitrary non-empty point set and $p\in \cP$. A \emph{decomposition (at $p$)} is a list of subsets $\cP_i\subseteq \cP$, 
  where $1\le i\le l$, whose pairwise intersection equals $p$ and the vertices of each square containing $p$ are contained in one of the $\cP_i$. 
  The integer $l\ge 1$ is called the cardinality of the decomposition. 
  
  We speak of a \emph{$2$-extension maximal decomposition} if 
  \begin{itemize}
    \item for each square $s$ of $\cP$ the intersection of the set of vertices of $s$ with each $\cP_i$ has a cardinality in $\{0,1,4\}$ and\\[-7mm]
    \item all $\cP_i$ can be obtained by recursive $2$-extension starting from the unit square.
  \end{itemize}   
\end{definition}

\begin{example}
  Let $\cP$ be the point set from Figure~\ref{fig_four_imply_five}. Slightly abusing notation we will use the depicted labels of the points 
  instead of their coordinates. For vertex $p=2$, a decomposition of $\cP$ at $p$ is given by $\cP_1=\{2,1,3,4\}$, $\cP_2=\{2,5,9,10\}$ 
  with cardinality $2$. Another decomposition of $\cP$ at $p=2$ is given by $\cP_1=\{2,1,3,4,5,9,10\}$ with cardinality $1$. Note that both decompositions (at $p=2$) 
  are also decompositions (at $p=2$) of $\cP'=\{1,2,3,4,5,9,10\}\subset\cP$. The latter decomposition is not a $2$-extension maximal decomposition 
  of $\cP'$ at $p=2$ since $\cP_1$ cannot be obtained by recursive $2$-extension starting from the unit square. The first decomposition, the one with cardinality $2$, 
  is indeed a $2$-extension maximal decomposition of $\cP'$ at $p=2$.
  
  Now let $\cP''=\cP\backslash\{1\}$ and $p=9$. Two decompositions of $\cP''$ (or $\cP$) at $p=9$ are given by 
  $\cP_1=\{9,3,4,6,7,8\}$, $\cP_2=\{9,2,5,10\}$ and $\cP_1=\{9,2,3,4,5,6,7,8,10\}$, respectively. Note that e.g.\ $\cP_1=\{9,3,7,8\}$, $\cP_2=\{9,3,4,6\}$, $\cP_3=\{9,2,5,10\}$ 
  is not a decomposition of $\cP''$ (or $\cP$) at $p=9$. The first decomposition is not a $2$-extension maximal decomposition of $\cP''$ at $p=9$ while the second 
  is. The unique $2$-extension maximal decomposition of $\cP$ at $p=9$ is given by $\cP$ itself.   
\end{example}

\begin{definition}
  \label{def_neighborhood_sq}
  Let $\cP\subset\R^2$ be an arbitrary non-empty point set and $p\in \cP$. The \emph{neighborhood of $p$ in $\cP$} is the smallest subset $\cP'\subseteq \cP$ that contains 
  $p$ and all vertices of squares of $\cP$ that contain $p$.
\end{definition}

\begin{example}
  Let $\cP$ be the point set from Figure~\ref{fig_four_imply_five}. Slightly abusing notation we will use the depicted labels of the points 
  instead of their coordinates. The neighborhood of $p=9$ in $\cP$ is given by $\{2,3,4,5,6,7,8,9,10\}=\cP\backslash\{1\}$ and the 
  neighborhood of $p=2$ in $\cP$ is given by $\{1,2,3,4,5,9,10\}$. 
\end{example}

Let $\cP$ be the neighborhood of a given vertex $p\in\cP$, i.e., every vertex in $\cP\backslash \{p\}$ is contained in a square that contains $p$ as a vertex. 
With this, the \emph{neighborhood graph} $\cG$ (of $p$) consists of the vertices of $\cP$ except the {\lq\lq}root vertex{\rq\rq} $p$. Two vertices $x,y$ in $\cG$ form an 
edge $\{x,y\}$ (in the graph theory sense) iff $p$, $x$, and $y$ are the vertices of a square in $\cP$. Note that the square corresponding to an edge in $\cG$ 
is indeed unique (for each edge). Let $C_1,\dots, C_r$ be the connected components of $\cG$. By $\cP_1,\dots,\cP_r$ we denote the subsets of $\cP$  
such that the points in $\cP_i$ are given by $p$ and the vertices of $C_i$. So, every square of $\cP$ that contains $p$ as a vertex is contained in exactly one of the 
point sets $\cP_i$.

\begin{lemma}
  \label{lemma_connected_components_neighborhood_graph}
  Using the above notation, every $\cP_i$ can be obtained by recursive $2$-extension starting from every square containing vertex $p$. Moreover, 
  $\cP_1,\dots,\cP_r$ is a decomposition of $\cP$ at $p$.
\end{lemma}
\begin{proof}
  Let $y\in\cP_i\backslash\{p\}$ be an arbitrary vertex and $s$ and arbitrary square in $\cP_i$ that contains $p$ as a vertex. With this, let $x\neq p$ be an arbitrary 
  vertex of the square $s$ and consider a path $\left(x_0,\dots,x_l\right)$ in $C_i$ with $x_0=x$ and $x_l=y$. For $0\le i<l$ the edge $\left\{x_i,x_{i+1}\right\}$ 
  corresponds to a square $s_i$ in $\cP_i$ with vertices $p$, $x_i$, and $x_{i+1}$. W.l.o.g.\ we assume that the vertex $x$ and the path in $C_i$ are chosen in such a 
  way such that $s=s_0$. Now we observe that we can reach the square $s_{i}$ from the square $s_{i-1}$ by a $2$-extension step for all $0<i<l$. (It may happen that 
  $s_{i-1}=s_i$.)   
  
  It remains to check that the conditions of Definition~\ref{def_decomposition} are satisfied for $\cP_1,\dots,\cP_r$.
\end{proof}

For an arbitrary point set $\cP\subset \R^2$ and an arbitrary point $p\in \cP$ consider the neighborhood $\cP'$ of $p$ in $\cP$. Let $\cP_1,\dots,\cP_r$ be 
a decomposition of $\cP'$ at $p$ according to Lemma~\ref{lemma_connected_components_neighborhood_graph}. The possible candidate for the $\cP_i$ are enumerated in 
Table~\ref{table_iso_types_exhaustive_2_corner} including the information of the maximum possible degree of $p$ in $\cP_i$.

\begin{table}[htp!]
  \begin{center}
  {%%\footnotesize
    \begin{tabular}{lllllllllllllllllll}
      \hline
      $n$             & 4 & 6 & 7 & 8 & 8 &  9 & 9 & 9 & 10 & 10 & 10 & 10 & 11 & 11 & 11 & 11 & 12 & 12 \\ 
      $m$             & 1 & 2 & 3 & 3 & 4 &  4 & 5 & 6 & 4  & 5  & 6  & 7  & 5  & 6  & 7  & 8  & 5  & 6  \\ 
      \#              & 1 & 2 & 2 & 5 & 1 & 12 & 1 & 1 & 11 & 10 & 5  & 1  & 79 & 14 & 3  & 2  & 26 & 79 \\
      $\delta_{\max}$ & 1 & 2 & 3 & 3 & 3 &  4 & 4 & 4 & 4  & 5  & 5  & 5  & 5  & 5  & 6  & 6  & 5  & 6  \\
      \hline
    \end{tabular}  
    \begin{tabular}{lllllllllllllllllll}
      \hline
      $n$             & 12 & 12 & 12 & 13  & 13  & 13 & 13 & 13 & 13 & 14 & 14  & 14  & 14  & 14 & 14 & 14 \\  
      $m$             & 7  & 8  & 9  & 6   & 7   & 8  & 9  & 10 & 11 & 6  & 7   & 8   & 9   & 10 & 11 & 12 \\ 
      \#              & 18 & 10 & 2  & 398 & 159 & 41 & 11 & 4  & 2  & 64 & 533 & 251 & 131 & 42 & 4  & 4  \\
      $\delta_{\max}$ & 6  & 6  & 7  & 6   & 7   & 7  & 7  & 8  & 8  & 6  & 7   & 7   & 8   & 8  & 9  & 9  \\
      \hline
    \end{tabular}
    \begin{tabular}{lllllllllllllllllll}
      \hline
      $n$             & 15   & 15   & 15  & 15  & 15 & 15 & 15 & 15 & 15 & 16  & 16   & 16   & 16   \\   
      $m$             & 7    & 8    & 9   & 10  & 11 & 12 & 13 & 14 & 15 & 7   & 8    & 9    & 10   \\ 
      \#              & 1594 & 1191 & 500 & 202 & 77 & 41 & 8  & 4  & 1  & 159 & 2812 & 2146 & 1204 \\
      $\delta_{\max}$ & 7    & 8    & 8   & 8   & 9  & 10 & 10 & 8  & 7  & 7   & 8    & 9    & 9    \\ 
      \hline
    \end{tabular}
    \begin{tabular}{lllllllllllllllllll}
      \hline
      $n$             & 16  & 16  & 16 & 16 & 16 & 16 & 16 & 16 & 17   & 17   & 17   & 17   & 17   \\   
      $m$             & 11  & 12  & 13 & 14 & 15 & 16 & 17 & 18 & 8    & 9    & 10   & 11   & 12   \\ 
      \#              & 591 & 160 & 87 & 25 & 3  & 3  & 3  & 1  & 5539 & 6358 & 4130 & 2099 & 1107 \\ 
      $\delta_{\max}$ & 9   & 10  & 10 & 11 & 9  & 9  & 8  & 8  & 8    & 9    & 9    & 10   & 10   \\ 
      \hline
    \end{tabular}
    \begin{tabular}{lllllllllllllllllll}
      \hline
      $n$             & 17  & 17  & 17  & 17 & 17 & 17 & 17 & 17 & 17 & 17 & 18  & 18    & 18    \\   
      $m$             & 13  & 14  & 15  & 16 & 17 & 18 & 19 & 20 & 21 & 22 & 8   & 9     & 10    \\ 
      \#              & 528 & 224 & 121 & 40 & 11 & 11 & 3  & 3  & 0  & 1  & 392 & 12293 & 12568 \\ 
      $\delta_{\max}$ & 11  & 11  & 12  & 12 & 10 & 10 & 7  & 8  & 0  & 8  & 8   & 9     & 10    \\ 
      \hline
    \end{tabular}
    \begin{tabular}{lllllllllllllllllll}
      \hline
      $n$             & 18   & 18   & 18   & 18   & 18  & 18  & 18  & 18 & 18 & 18 & 18 & 18 & 18 & 18 & 18 \\   
      $m$             & 11   & 12   & 13   & 14   & 15  & 16  & 17  & 18 & 19 & 20 & 21 & 22 & 23 & 24 & 25 \\ 
      \#              & 8840 & 5276 & 2272 & 1223 & 480 & 227 & 102 & 63 & 29 & 19 & 7  & 5  & 2  & 0  & 1  \\    
      $\delta_{\max}$ & 10   & 10   & 11   & 12   & 12  & 13  & 13  & 11 & 11 & 11 & 9  & 10 & 9  & 0  & 9  \\ 
      \hline
    \end{tabular}}
    \caption{Number of non-similar neighborhood points sets $\cP$ (of vertex $1$) that are produced by recursive $2$-extension starting from a unit square (and including vertex $1$ in 
    each step) per number of points $n$ and squares $m=S_{\square}(\cP)$.}
    \label{table_iso_types_exhaustive_2_corner}
  \end{center}
\end{table}

\begin{theorem}
  \label{thm_exact_values_squares_2}
  We have $S_{\square}(14)=15$, $S_{\square}(15)=17$, $S_{\square}(16)=20$, and $S_{\square}(17)=22$.
\end{theorem}
\begin{proof}
  Examples attaining the stated number of squares are given in  Appendix~\ref{appendix_dissimilar_blbk_point_sets_squares}. 
  Now let $\cP$ be an $n$-point set with $n\in\{14,15,16,17\}$ and more squares than proposed in our statement. Let $\widetilde{\cP}$ be a 
  $2$-extension maximal subconfiguration of maximum cardinality. From Lemma~\ref{lemma_2_extension_maximal_mixed_triangle_square} 
  and Lemma~\ref{lemma_ub_mixed_triangle_square} we conclude $\#\widetilde{\cP}\le 6$ if $n\in\{14,15\}$ and $\#\widetilde{\cP}\le 7$ if $n\in\{16,17\}$. 
  From Lemma~\ref{lemma_minimum_degree} we conclude $\delta_{\max}\ge 5$ if $n\in\{14,15\}$ and $\delta_{\max}\ge 6$ if $n\in\{16,17\}$. 
  Now let $p$ be a vertex with maximum degree in $\cP$ and $\cP'$ be the neighborhood of $p$ in $\cP$. Let $\cP_1,\dots,\cP_r$ be 
  a decomposition of $\cP'$ at $p$ according to Lemma~\ref{lemma_connected_components_neighborhood_graph}. W.l.o.g.\ we assume $\#\cP_1 \ge \dots\ge\#\cP_r$. 
  Additionally, we have the conditions $\#\cP_i\le \#\widetilde{\cP}$ for $1\le i\le r$, $1+\sum_{i=1}^r (\#\cP_i-1)\le n$, and $\sum_{i=1}^r \delta_{\max}(\cP_i)\ge \delta_{\max}(\cP)$. 
  This leaves only very few choices for $(\#\cP_1,\dots,\#\cP_r)$, which we enumerate in the following.
  \begin{itemize}
    \item $n=14$: $(6,6,4)$;\\[-7mm]
    \item $n=15$: $(6,6,4)$;\\[-7mm]
    \item $n=16$: $(7,7)$, $(7,7,4)$, $(6,6,6)$;\\[-7mm]
    \item $n=17$: $(7,7)$, $(7,7,4)$, $(6,6,6)$.
  \end{itemize}   
  If there would exist indices $1\le i,j\le r$ with $i\neq j$ and a square $s$ in $\cP'$ such that at least two vertices on $s$ are contained in $\cP_i$ and 
  at least one vertex of $s$ would be contained in $\cP_j$, then $\cP_i\cup\cP_j$ can be obtained by recursive $2$-extension starting from the unit square. 
  However, since $\#(\cP_i\cup\cP_j)=\#\cP_i+\cP_j-1$ this would yield a contradiction to the maximum possible size of $\widetilde{\cP}$. Thus, we 
  have $S_{\square}(\cP')\le\sum_{i=1}^n S_{\square}(\#\cP_i)$.
  
  For $n=14$ we compute $S_{\square}(\cP)=S_{\square}(\cP')\le 5<16$. For $n=15$ we have $\delta_{\max}=5$ and $S_{\square}(\cP')\le 5$. The unique vertex outside 
  of $\cP'$ can have degree at most $5$ so that $S_{\square}(\cP)\le 5+5=10<18$. For $n=16$ and $(\#\cP_1,\dots,\#\cP_r)=(7,7,4)$ or $(6,6,6)$ we have 
  $S_{\square}(\cP)=S_{\square}(\cP')\le 7<21$. For $n=16$ the only remaining case is $(\#\cP_1,\dots,\#\cP_r)=(7,7)$ with $\delta_{\max}=6$. Here, each square $s$ 
  not contained in $\cP'$ can have at most one vertex with $\cP_1$ or $\cP_2$ in common since we would obtain a contradiction to $\#\widetilde{P}\le 7$. Thus each such 
  square uses at least two out of the three vertices in $\cP\backslash\cP'$, so that 
  $$
    S_{\square}(\cP)\le S_{\square}(\cP')+\frac{3\cdot\delta_{\max}}{2}\le 6+9=15<21.
  $$    
  For $n=17$ we have $\delta_{\max}\le 7$. If $(\#\cP_1,\dots,\#\cP_r)=(7,7,4)$ or $(6,6,6)$, then we have $S_{\square}(\cP')\le 7$, so that
  $$
    S_{\square}(\cP)\le S_{\square}(\cP') +(\#\cP-\#\cP')\cdot\delta_{\max}\le 7+7=14<23.
  $$
  For $n=17$ the only remaining case is $(\#\cP_1,\dots,\#\cP_r)=(7,7)$ with $\delta_{\max}=6$. Here, each square $s$ 
  not contained in $\cP'$ can have at most one vertex with $\cP_1$ or $\cP_2$ in common since we would obtain a contradiction to $\#\widetilde{P}\le 7$. Thus each such 
  square uses at least two out of the four vertices in $\cP\backslash\cP'$, so that 
  $$
    S_{\square}(\cP)\le S_{\square}(\cP')+\frac{4\cdot\delta_{\max}}{2}\le 6+12=18<23.
  $$
\end{proof}

Of course we can apply an incomplete $2$-extension approach considering point sets with {\lq\lq}many{\rq\rq} squares for the next iteration only, see Table ~\ref{table_lower_bounds_number_of_contained_squares} 
and Table~\ref{table_lower_bounds_number_of_contained_squares2} (in Appendix~\ref{appendix_dissimilar_blbk_point_sets_squares}).
\begin{table}[htp]
\begin{center}
\begin{tabular}{lllllllllllllllllll}
\hline
  $n$ &                   4 & 5 & 6 & 7 & 8 & 9 & 10 & 11 & 12 & 13 & 14 & 15 & 16 & 17 & 18 & 19 & 20 & 21 \\
  $S_{\square}(n)\ge$ &   1 & 1 & 2 & 3 & 4 & 6 &  7 &  8 & 11 & 13 & 15 & 17 & 20 & 22 & 25 & 28 & 32 & 37 \\  
\hline
\end{tabular}
\begin{tabular}{lllllllllllllllllll}
\hline
  $n$ &                  22 & 23 & 24 & 25 & 26 & 27 & 28 & 29 & 30 & 31 & 32 & 33 & 34 &  35 & 36 &  37 \\
  $S_{\square}(n)\ge$ &  40 & 43 & 47 & 51 & 56 & 60 & 65 & 70 & 75 & 81 & 88 & 92 & 97 & 103 & 109 & 117 \\
\hline
\end{tabular}
\begin{tabular}{llllllllllllllllll}
\hline
  $n$ &                   38 & 39 & 40 & 41 & 42 & 43 & 44 & 45 & 46 & 47 & 48 & 49 & 50 \\
  $S_{\square}(n)\ge$ &  123 & 130 & 137 & 144 & 151 & 158 & 166 & 175 & 182 & 189 & 198 & 207 & 216\\
\hline
\end{tabular}
\caption{Lower bounds for $S_{\square}(n)$ for $4\le n\le 50$.}
\label{table_lower_bounds_number_of_contained_squares}
\end{center}
\end{table}

\section*{Acknowledgments}
The present investigation is inspired by an extensive discussion on the Sequence Fans Mailing List with contributions from many people. 
We especially thank Beno\^{i}t Jubin, Peter Munn, Hugo van der Sanden, and N.J.A.~Sloane for many remarks on a predecessor of the present text. 
Proposition~\ref{prop_grid_size_bound} arose from a fruitful discussion with Warren D.~Smith.

%% \bibliography{squares_in_plane_point_sets}
%% \bibdata{squares_in_plane_point_sets}
%% \bibliographystyle{amsplain}  
%% %%\bibliographystyle{abbrv}

\providecommand{\bysame}{\leavevmode\hbox to3em{\hrulefill}\thinspace}
\providecommand{\MR}{\relax\ifhmode\unskip\space\fi MR }
\providecommand{\MRhref}[2]{%
  \href{http://www.ams.org/mathscinet-getitem?mr=#1}{#2}
}
\providecommand{\href}[2]{#2}

\appendix
\section*{Appendix}
Here we list the known equivalence classes of $n$-point sets attaining the best lower bounds known \texttt{blbk} for 
$S_{\cQ}(n)$ for different variants of $\cQ$. All of our examples are situated on subsets of the 
integer grid $\Z^2$. By an {\lq\lq}x{\rq\rq} we denote a chosen point and by {\lq\lq}.{\rq\rq} a grid 
point that is not part of the point set. The occurring numbers just state the numbers of points and we separate
different number of points by {\lq\lq}$\vert${\rq\rq}.

\section{Dissimilar point sets for isosceles right triangles}
\label{appendix_dissimilar_blbk_point_sets_rit}

Here we are interested in the (known) similarity classes of point sets that attain $S_{\rit}(n)$ for $n\le 14$ or 
match the lower bounds stated in Table~\ref{table_lower_bounds_number_of_contained_irt}. The number of 
similarity classes is listed in Table~\ref{table_blbk_similarity_classes_rit}.  

\begin{table}[htp!]
  \begin{center}
    \begin{tabular}{llllllllllllllllllll}
      \hline
      $n$  & 3 & 4 & 5 & 6 & 7 & 8 & 9 & 10 & 11 & 12 & 13 & 14 & 15 & 16 & 17 & 18 & 19 & 20 & 21 \\ 
      $\#$ & 1 & 1 & 1 & 1 & 2 & 5 & 1 & 1  & 1  & 2  & 2  & 1  & 1  & 1  & 1  & 2  & 2  & 1  & 1  \\ 
      \hline 
    \end{tabular}
    \begin{tabular}{llllllllllllllllllll}
      \hline
      $n$  & 22 & 23 & 24 & 25 & 26 & 27 & 28 & 29 & 30 & 31 & 32 & 33 & 34 & 35 & 36 & 37 & 38 & 39 \\ 
      $\#$ & 2  & 1  & 1  & 1  & 1  & 3  & 2  & 2  & 1  & 2  & 1  & 1  & 1  & 1  & 1  & 1  & 1  & 1  \\ 
      \hline 
    \end{tabular}
    \begin{tabular}{llllllllllllllllllll}
      \hline
      $n$  & 40 & 41 & 42 & 43 & 44 & 45 & 46 & 47 & 48 & 49 & 50 \\ 
      $\#$ & 1  & 1  & 1  & 2  & 2  & 3  & 3  & 1  & 2  & 1  & 1  \\ 
      \hline 
    \end{tabular}
    \caption{Number of known dissimilar point sets attaining \texttt{blbk} for isosceles right triangles.}
    \label{table_blbk_similarity_classes_rit}
  \end{center}
\end{table}

\begin{verbatim}
   |    |     |     |          |                x.x.      |     |      
   |    | .x. | xx. | xxx x.x. | x.x.x xxx. xxx .xxx .xx. | xxx | xxx. 
xx | xx | xxx | xxx | xxx .xxx | .x.x. .xxx xxx x.x. xxxx | xxx | xxxx 
x. | xx | .x. | .x. | .x. x.x. | x.x.x x.x. xx. .x.. .xx. | xxx | xxx. 
3  | 4  | 5   | 6   | 7        | 8                        | 9   | 10   
\end{verbatim}
\begin{verbatim}
     |            |             |        |        |         | ...x... 
     |            |             |        | .x.x.. | ..x.x.. | ..x.x.. 
     |            | ..x.. .x.x. | .x.x.. | x.x.x. | .x.x.x. | .x.x.x.
xxx. | .xxx. .xx. | .xxx. x.x.x | x.x.x. | .xxx.x | x.xxx.x | x.xxx.x
xxxx | xxxxx xxxx | xxxxx .xxx. | .xxx.x | x.x.x. | .x.x.x. | .x.x.x.
xxx. | .xxx. xxxx | .xxx. x.x.x | x.x.x. | .x.x.. | ..x.x.. | ..x.x..
.x.. | ..x.. .xx. | ..x.. .x.x. | .x.x.. | ..x... | ...x... | ...x...
11   | 12         | 13          | 14     | 15     | 16      | 17
\end{verbatim}
\begin{verbatim}
       .x.x... |       .x.x.x. |       |       | ..x.x..        | ..x.x..
.x.x.. ..x.x.. |       ..x.x.. |       |       | .x.x.x.        | .x.x.x.
x.x.x. .x.x.x. | .xxx. .x.x.x. | .xxx. | .xxx. | x.xxx.x .xxx.. | x.xxx.x
.xxx.x x.xxx.x | xxxxx x.xxx.x | xxxxx | xxxxx | .x.x.x. xxxxx. | .xxx.x.
x.xxx. .x.x.x. | xxxx. .x.x.x. | xxxxx | xxxxx | x.x.x.x xxxxxx | x.x.x.x
.x.x.x ..x.x.. | xxxxx ..x.x.. | xxxxx | xxxxx | .x.x.x. xxxxx. | .x.x.x.
..x.x. ...x... | .x.x. ...x... | .x.x. | .xxx. | ..x.x.. .xxx.. | ..x.x..
18             | 19            | 20    | 21    | 22             | 23
\end{verbatim}

\pagebreak

\begin{verbatim}
        |         |          |          ..x.x...          | ..x.x... .x.x.x..
..x.x.. | ..x.x.. | ..x.x... | .x.x.x.. .x.x.x.. .x.x.x.. | .x.x.x.. x.x.x.x.
.x.x.x. | .x.x.x. | .x.x.x.. | ..x.x.x. x.xxx.x. x.x.x.x. | x.x.x.x. .x.x.x.x
x.xxx.x | x.xxx.x | x.xxx.x. | .x.xxx.x .xxxxx.x .x.x.x.x | .x.xxx.x x.xxx.x.
.xxxxx. | .xxxxx. | .xxxxx.x | x.xxxxx. x.xxx.x. x.xxx.x. | x.xxx.x. .x.x.x.x
x.x.x.x | x.xxx.x | x.xxx.x. | .x.xxx.x .x.x.x.. .x.x.x.x | .x.x.x.x x.x.x.x.
.x.x.x. | .x.x.x. | .x.x.x.. | ..x.x.x. ..x.x... x.x.x.x. | ..x.x.x. .x.x.x..
..x.x.. | ..x.x.. | ..x.x... | ...x.x.. ...x.... .x.x.x.. | ...x.x.. ....x...
24      | 25      | 26       | 27                         | 28
\end{verbatim}
\begin{verbatim}
                   |           |           ...x.x... | ...x.x... | ...x.x...
.x.x.x.. ..x.x.x.. | ..x.x.x.. | ..x.x.x.. ..x.x.x.. | ..x.x.x.. | ..x.x.x..
x.x.x.x. .x.x.x.x. | .x.x.x.x. | .x.x.x.x. .x.x.x.x. | .x.x.x.x. | .x.x.x.x.
.x.x.x.x ..x.x.x.. | x.x.x.x.. | x.x.x.x.x x.x.x.x.x | x.x.x.x.x | x.x.x.x.x
x.xxx.x. .x.xxx.x. | .x.xxx.x. | .x.xxx.x. .x.xxx.x. | .x.xxx.x. | .x.xxx.x.
.x.x.x.x x.x.x.x.x | x.x.x.x.x | x.x.x.x.x x.x.x.x.. | x.x.x.x.x | x.x.x.x.x
x.x.x.x. .x.x.x.x. | .x.x.x.x. | .x.x.x.x. .x.x.x.x. | .x.x.x.x. | .x.x.x.x.
.x.x.x.. ..x.x.x.. | ..x.x.x.. | ..x.x.x.. ..x.x.x.. | ..x.x.x.. | ..x.x.x..
..x.x... ...x.x... | ...x.x... | ...x.x... .....x... | ...x..... | ...x.x...
29                 | 30        | 31                  | 32        | 33
\end{verbatim}
\begin{verbatim}
...x.x... |         |         |         | ..x.x.x.. | ..x.x.x.. | ..x.x.x..
..x.x.x.. |         |         |         | .x.x.x.x. | .x.x.x.x. | .x.x.x.x.
.x.x.x.x. | ..xxx.. | ..xxx.. | ..xxx.. | x.x.x.x.x | x.x.x.x.x | x.x.x.x.x
x.xxx.x.x | .xxxxx. | .xxxxx. | .xxxxx. | .x.xxx.x. | .x.xxx.x. | .x.xxx.x.
.x.xxx.x. | xxxxxxx | xxxxxxx | xxxxxxx | x.x.x.x.x | x.xxx.x.x | x.xxxxx.x
x.x.x.x.x | xxxxxx. | xxxxxxx | xxxxxxx | .x.x.x.x. | .x.x.x.x. | .x.x.x.x.
.x.x.x.x. | xxxxxxx | xxxxxxx | xxxxxxx | x.x.x.x.x | x.x.x.x.x | x.x.x.x.x
..x.x.x.. | .xxxxx. | .xxxxx. | .xxxxx. | .x.x.x.x. | .x.x.x.x. | .x.x.x.x.
...x.x... | ..x.x.. | ..x.x.. | ..xxx.. | ..x.x.x.. | ..x.x.x.. | ..x.x.x..
34        | 35      | 36      | 37      | 38        | 39        | 40
\end{verbatim}
\begin{verbatim}
          |            |            ..x.x.x... | ..x.x.x... ..x.x.x...
..x.x.x.. | ..x.x.x... | ..x.x.x... .x.x.x.x.. | .x.x.x.x.. .x.x.x.x..
.x.x.x.x. | .x.x.x.x.. | .x.x.x.x.. x.x.x.x.x. | x.x.x.x.x. x.x.x.x.x.
x.x.x.x.x | x.x.x.x.x. | x.x.x.x.x. .x.xxx.x.. | .x.x.x.x.x .x.xxx.x.x
.x.xxx.x. | .x.xxx.x.x | .x.xxx.x.x x.xxxxx.x. | x.x.xxx.x. x.xxxxx.x.
x.xxxxx.x | x.xxxxx.x. | x.xxxxx.x. .x.xxx.x.x | .x.xxx.x.x .x.xxx.x.x
.x.xxx.x. | .x.xxx.x.. | .x.xxx.x.x x.x.x.x.x. | x.x.x.x.x. x.x.x.x.x.
x.x.x.x.x | x.x.x.x.x. | x.x.x.x.x. .x.x.x.x.. | .x.x.x.x.. .x.x.x.x..
.x.x.x.x. | .x.x.x.x.. | .x.x.x.x.. ..x.x.x... | ..x.x.x.x. ..x.x.x...
..x.x.x.. | ..x.x.x... | ..x.x.x... .....x.... | ...x.x.... .....x....
41        | 42         | 43                    | 44
\end{verbatim}

\pagebreak

\begin{verbatim}
                               |            ....x.x....
...x.x.x... ..x.x.x...         | ..x.x.x... ...x.x.x... ...x.x.x... 
..x.x.x.x.. .x.x.x.x..         | .x.x.x.x.. ..x.x.x.x.. ..x.x.x.x..
.x.x.x.x.x. x.x.x.x.x.         | x.x.x.x.x. .x.x.x.x.x. .x.x.x.x.x.
..x.xxx.x.. .x.xxx.x.x .xxxxx. | .x.x.x.x.x x.x.xxx.x.x x.x.xxx.x..
.x.xxxxx.x. x.xxxxx.x. xxxxxxx | x.x.xxx.x. .x.x.x.x.x. .x.xxxxx.x.
x.x.xxx.x.x .x.xxx.x.x xxxxxxx | .x.xxx.x.x x.x.x.x.x.x x.x.xxx.x.x
.x.x.x.x.x. x.x.x.x.x. xxxxxxx | x.x.x.x.x. .x.x.x.x.x. .x.x.x.x.x.
..x.x.x.x.. .x.x.x.x.. xxxxxxx | .x.x.x.x.x ..x.x.x.x.. ..x.x.x.x..
...x.x.x... ..x.x.x... xxxxxxx | ..x.x.x.x. ...x.x.x... ...x.x.x...
....x.x.... ...x.x.... .xxxxx. | ...x.x.x.. ....x.x.... ....x.x.... 
45                             | 46
\end{verbatim}
\begin{verbatim}
....x.x.... | ....x.x.... ....x.x.... | ....x.x.... | ..x.x.x....
...x.x.x... | ...x.x.x... ...x.x.x... | ...x.x.x... | ...x.x.x...
..x.x.x.x.. | ..x.x.x.x.. ..x.x.x.x.. | ..x.x.x.x.. | ..x.x.x.x..
.x.x.x.x.x. | .x.x.x.x.x. .x.x.x.x.x. | .x.x.x.x.x. | .x.x.x.x.x.
x.x.xxx.x.x | x.x.xxx.x.x x.x.xxx.x.x | x.x.xxx.x.x | x.x.xxx.x.x
.x.xxx.x.x. | .x.xxxxx.x. .x.xxxxx.x. | .x.xxxxx.x. | .x.xxxxx.x.
x.x.x.x.x.x | x.x.xxx.x.x x.x.x.x.x.x | x.x.xxx.x.x | x.x.xxx.x.x
.x.x.x.x.x. | .x.x.x.x.x. .x.x.x.x.x. | .x.x.x.x.x. | .x.x.x.x.x.
..x.x.x.x.. | ..x.x.x.x.. ..x.x.x.x.. | ..x.x.x.x.. | ..x.x.x.x..
...x.x.x... | ...x.x.x... ...x.x.x... | ...x.x.x... | ...x.x.x...
....x.x.... | ....x...... ....x.x.... | ....x.x.... | ....x.x....
47          | 48                      | 49          | 50
\end{verbatim}

%% 3->1, 4->4, 5->8, 6->11, 7->15, 8->20, 9->28, 10->35, 11->43, 12->52, 13->64, 14->74, 15->85, 16->97, 17->112, 18->124, 19->139, 20->156,       
%% 21->176, 22->192, 23->210, 24->229, 25->252, 26->271, 27->291, 28->314, 29->338, 30->363, 31->389, 32->417, 33->448, 34->473, 35->501,
%% 36->531, 37->564, 38->594, 39->626, 40->659, 41->696, 42->728, 43->763, 44->799, 45->836, 46->874, 47->914, 48->955, 49->1000, 50->1038  

\section{Dissimilar point sets attaining the value in Lemma~\ref{lemma_ub_mixed_triangle_square}}
\label{appendix_dissimilar_blbk_point_sets_mixed}

The enumeration of all $n$-point sets with $n\le 9$ that can be obtained by $1$-extension yields the following point sets 
attaining $S_{\rit-3\square}(n)$ as mentioned in Lemma~\ref{lemma_ub_mixed_triangle_square}:

\begin{verbatim}
   |     | xxx .x. | x.x. xxx. xxx x.x.x xxx.x | x.x.x  
xx | xxx | .x. xxx | .xxx .x.x .x. .xx.. .x... | .xxx.
x. | .x. | x.. .x. | ..x. x... x.x ..x.. ..x.. | ..x..
3  | 4   | 5       | 6                         | 7
\end{verbatim}
\begin{verbatim}
           |       | x.x.. .x...
x.x. x.x.. | x.x.. | .xxx. ..x.x
.xxx .xx.. | .xxx. | ..x.x .xxx.
..x. ..x.x | ..x.x | .x... x.x..
.x.. .x... | .x... | ..x.. ...x.
7          | 8     | 9
\end{verbatim}
We remark that the only possible case where $S_{\rit-3\square}(n)$ might be attained a point set that cannot be obtained by $1$-extension is given by 
$n=7$ and $\S_{\rit}(\cP)=10$. However, Lemma~\ref{lemma_1_extension_maximal_rit} then implies that $\cP$ contains a $5$-point set $\cP'$ with 
$S_{\rit}(\cP')=8$ as a subconfiguration. Since $\cP'$ contains a square, we have $S_{\rit-3\square}(\cP)\le 7$. Thus, the above list is exhaustive. 

\section{Dissimilar point sets attaining \texttt{blbk} for $S_{\square}(n)$}
\label{appendix_dissimilar_blbk_point_sets_squares}

\begin{table}[htp!]
  \begin{center}
    \begin{tabular}{llllllllllllllllllll}
      \hline
      $n$  & 4 & 6 & 7 & 8 & 9 & 10 & 11 & 12 & 13 & 14 & 15 & 16 & 17 & 18 & 19 & 20 & 21 & 22 & 23 \\ 
      $\#$ & 1 & 2 & 2 & 2 & 1 & 1  & 5  & 1  & 1  & 2  & 2  & 2  & 4  & 3  & 5  & 1  & 1  & 1  & 3  \\ 
      \hline 
    \end{tabular}
    \begin{tabular}{llllllllllllllllllll}
      \hline
      $n$  & 24 & 25 & 26 & 27 & 28 & 29 & 30 & 31 & 32 & 33 & 34 & 35 & 36 & 37 & 38 & 39 & 40 & 41 \\ 
      $\#$ & 1  & 1  & 1  & 2  & 2  & 1  & 4  & 1  & 1  & 2  & 4  & 2  & 3  & 1  & 1  & 2  & 1  & 2  \\
      \hline 
    \end{tabular}
    \begin{tabular}{llllllllllllllllllll}
      \hline
      $n$  & 42 & 43 & 44 & 45 & 46 & 47 & 48 & 49 & 50 & 51 & 52 & 53 & 54 & 55 & 56 & 57 & 58 & 59 \\ 
      $\#$ & 2  & 6  & 2  & 1  & 1  & 7  & 2  & 1  & 4  & 1  & 1  & 1  & 2  & 1  & 5  & 3  & 1  & 1  \\
      \hline 
    \end{tabular}
    \begin{tabular}{llllllllllllllllllll}
      \hline
      $n$  & 60 & 61 & 62 & 63 & 64 & 65 & 66 & 67 & 68 & 69 & 70 & 71 & 72 & 73 & 74 & 75 & 76 & 77 \\ 
      $\#$ & 1  & 1  & 9  & 3  & 2  & 2  & 1  & 4  & 1  & 1  & 2  & 4  & 2  & 8  & 1  & 3  & 2 &  1  \\
      \hline 
    \end{tabular}
    \begin{tabular}{llllllllllllllllllll}
      \hline
      $n$  & 78 & 79 & 80 & 81 & 82 & 83 & 84 & 85 & 86 & 87 & 88 & 89 & 90 & 91 & 92 & 93 & 94 & 95 \\ 
      $\#$ & 3  & 4  & 2  & 4  & 3  & 2  & 2  & 1  & 4  & 1  & 1  & 1  & 2  & 2  & 1  & 2  & 2  &  5 \\  
      \hline 
    \end{tabular}
    \begin{tabular}{llllllllllllllllllll}
      \hline
      $n$  & 96 & 97 & 98 & 99 & 100 \\ %% (905,101): 2; (922,102): 5
      $\#$ & 1  & 1  & 1  & 2  & 2   \\
      \hline 
    \end{tabular}
    \caption{Number of known pairwise dissimilar point sets attaining \texttt{blbk} for $S_{\square}(n)$.}
    \label{table_blbk_similarity_classes}
  \end{center}
\end{table}

Here we list the known equivalence classes of $n$-point sets attaining \texttt{blbk} for $S_{\square}(n)$ mentioned in 
Table~\ref{table_blbk_similarity_classes} for $n\le 100$. Here the numbers just state the numbers of points and we separate different 
number of points by {\lq\lq}$\vert${\rq\rq}. 
\begin{verbatim}
   |         |         |          |     |      | .xx. xxx.      xxx.        
   |     xx. | xxx xx. | .xx. xxx | xxx | xxx. | xxxx xxxx xxxx xxxx .xxx.  
xx | xxx xxx | xxx xxx | xxxx xxx | xxx | xxxx | xxxx xxx. xxxx xxx. xxxxx 
xx | xxx .x. | .x. .xx | .xx. xx. | xxx | xxx. | .x.. .x.. xxx. ..x. .xxx. 
4  | 6       | 7       | 8        | 9   | 10   | 11                        
\end{verbatim}
\begin{verbatim}
     |      |           |            |            | .xxx. .xxx. 
.xx. | xxx. | xxx. xxxx | .xxx. xxxx | .xxx. xxxx | xxxx. xxxxx xxxx. xxxx.  
xxxx | xxxx | xxxx xxxx | xxxxx xxxx | xxxxx xxxx | xxxxx xxxxx xxxxx xxxxx
xxxx | xxxx | xxxx xxxx | xxxx. xxxx | xxxxx xxxx | xxxx. .xxx. xxxxx xxxx.
.xx. | .xx. | .xxx .xx. | .xxx. xxx. | .xxx. xxxx | ..x.. ..x.. .xxx. xxxx.
12   | 13   | 14        | 15         | 16         | 17
\end{verbatim}
\begin{verbatim}
.xxx. xxxx.       | .xxx. .xxx. .xxx. .xxx. xxxx. | .xxx. | .xxx. | xxxx.
xxxxx xxxx. xxxx. | xxxxx xxxxx xxxxx xxxx. xxxxx | xxxxx | xxxxx | xxxxx
xxxxx xxxxx xxxxx | xxxxx xxxxx xxxxx xxxxx xxxxx | xxxxx | xxxxx | xxxxx
xxxx. xxxx. xxxxx | xxxx. xxxx. xxxxx .xxxx xxxx. | xxxxx | xxxxx | xxxxx
..x.. ..x.. xxxx. | ..xx. .xx.. ..x.. .xxx. ..x.. | .xx.. | .xxx. | .xxx.
18                | 19                            | 20    | 21    | 22
\end{verbatim}
\begin{verbatim}
                   |        |        |        | .xxxx. .xxxx. | .xxxx. .xxxx.
.xxxx. xxxx. xxxxx | .xxxx. | .xxxx. | .xxxx. | xxxxxx xxxxx. | xxxxx. xxxxxx
xxxxx. xxxxx xxxxx | xxxxx. | xxxxxx | xxxxxx | xxxxxx xxxxxx | xxxxxx xxxxxx
xxxxxx xxxxx xxxxx | xxxxxx | xxxxxx | xxxxxx | xxxxxx xxxxxx | xxxxxx xxxxxx
xxxxx. xxxxx xxxxx | xxxxx. | xxxxx. | xxxxxx | .xxxx. .xxxx. | xxxxx. .xxxx.
.xxx.. .xxxx .xxx. | .xxxx. | .xxxx. | .xxxx. | ..x... ..xx.. | ..xx.. ..xx.. 
23                 | 24     | 25     | 26     | 27            | 28
\end{verbatim}
\begin{verbatim}
.xxxx. | .xxxx. .xxxx. .xxxx. .xxxx. | .xxxx. | .xxxx. | .xxxx.. xxxxx.   
xxxxxx | xxxxxx xxxxxx xxxxxx xxxxx. | xxxxxx | xxxxxx | xxxxxx. xxxxxx 
xxxxxx | xxxxxx xxxxxx xxxxxx xxxxxx | xxxxxx | xxxxxx | xxxxxxx xxxxxx 
xxxxxx | xxxxxx xxxxxx xxxxxx xxxxxx | xxxxxx | xxxxxx | xxxxxx. xxxxxx 
xxxxx. | xxxxx. xxxxxx xxxxx. .xxxxx | xxxxxx | xxxxxx | xxxxxx. xxxxxx  
..xx.. | ..xxx. ..xx.. .xxx.. .xxxx. | .xxx.. | .xxxx. | .xxxx.. .xxxx.   
29     | 30                          | 31     | 32     | 33
\end{verbatim}
\begin{verbatim}
        .xxxx..                 |         .xxxx.. |         .xxxx.. ..xxx..    
.xxxxx. xxxxxx. .xxxxx. .xxxx.. | .xxxxx. xxxxxx. | .xxxxx. xxxxxx. .xxxxx. 
xxxxxx. xxxxxx. xxxxxx. xxxxxx. | xxxxxx. xxxxxx. | xxxxxx. xxxxxxx xxxxxxx
xxxxxx. xxxxxxx xxxxxxx xxxxxxx | xxxxxxx xxxxxxx | xxxxxxx xxxxxxx xxxxxxx
xxxxxxx xxxxxx. xxxxxx. xxxxxxx | xxxxxxx xxxxxx. | xxxxxxx xxxxxx. xxxxxxx
xxxxxx. .xxxx.. xxxxxx. xxxxxx. | xxxxxx. .xxxxx. | xxxxxx. .xxxxx. .xxxxx. 
.xxxx.. ...x... .xxxx.. .xxxx.. | .xxxx.. ...x... | .xxxxx. ...x... ..xx... 
34                              | 35              | 36 
\end{verbatim}
\begin{verbatim}
..xxx.. | .xxxx.. | .xxxxx. .xxxx.. | .xxxxx. | .xxxxx. .xxxxx. | .xxxxx. ..xxxx..
.xxxxx. | .xxxxx. | .xxxxx. xxxxxx. | xxxxxx. | xxxxxxx xxxxxx. | xxxxxxx .xxxxxx.
xxxxxxx | xxxxxxx | xxxxxxx xxxxxxx | xxxxxxx | xxxxxxx xxxxxxx | xxxxxxx xxxxxxx.
xxxxxxx | xxxxxxx | xxxxxxx xxxxxxx | xxxxxxx | xxxxxxx xxxxxxx | xxxxxxx xxxxxxxx
xxxxxxx | xxxxxxx | xxxxxxx xxxxxxx | xxxxxxx | xxxxxxx xxxxxxx | xxxxxxx xxxxxxx.
.xxxxx. | .xxxxx. | .xxxxx. .xxxxx. | .xxxxx. | .xxxxx. xxxxxx. | xxxxxx. .xxxxxx.
..xxx.. | ..xxx.. | ..xxx.. ..xxx.. | ..xxx.. | ..xxx.. ..xxx.. | ..xxx.. ..xxxx..
37      | 38      | 39              | 40      | 41              | 42
\end{verbatim}
\begin{verbatim}
..xxxx.. .xxxxx.. .xxxxx. .xxxxx. .xxxxx. .xxxxx. | .xxxxx. ..xxxx.. | .xxxxx.
.xxxxxx. .xxxxxx. xxxxxxx xxxxxxx xxxxxxx xxxxxx. | xxxxxxx .xxxxxx. | xxxxxxx
xxxxxxxx xxxxxxx. xxxxxxx xxxxxxx xxxxxxx xxxxxxx | xxxxxxx xxxxxxxx | xxxxxxx
xxxxxxxx xxxxxxxx xxxxxxx xxxxxxx xxxxxxx xxxxxxx | xxxxxxx xxxxxxxx | xxxxxxx
xxxxxxx. xxxxxxx. xxxxxxx xxxxxxx xxxxxxx xxxxxxx | xxxxxxx xxxxxxxx | xxxxxxx
.xxxxxx. .xxxxxx. xxxxxx. xxxxxxx xxxxxx. .xxxxxx | xxxxxxx .xxxxxx. | xxxxxxx
..xxxx.. ..xxxx.. ..xxxx. ..xxx.. .xxxx.. .xxxxx. | .xxxx.. ..xxxx.. | .xxxxx.
43                                                | 44               | 45 
\end{verbatim}
\begin{verbatim}
         | ..xxxx.. ..xxxx.. ..xxxx.. .xxxxx..                    .xxxxx..  
.xxxxx.. | .xxxxxx. .xxxxxx. .xxxxxx. xxxxxxx. .xxxxx.. ..xxxxx.. xxxxxxx.
xxxxxxx. | xxxxxxxx xxxxxxx. xxxxxxx. xxxxxxx. xxxxxxx. .xxxxxxx. xxxxxxx.
xxxxxxx. | xxxxxxxx xxxxxxxx xxxxxxxx xxxxxxxx xxxxxxxx .xxxxxxx. xxxxxxxx
xxxxxxxx | xxxxxxxx xxxxxxxx xxxxxxxx xxxxxxx. xxxxxxxx xxxxxxxxx xxxxxxx.
xxxxxxx. | .xxxxxx. .xxxxxx. xxxxxxx. xxxxxxx. xxxxxxx. .xxxxxxx. xxxxxxx.
xxxxxxx. | .xxxxxx. .xxxxxx. .xxxxx.. .xxxxx.. xxxxxxx. .xxxxxxx. .xxxxx..
.xxxxx.. | ...x.... ...xx... ...xx... ....x... .xxxxx.. ..xxxxx.. ...x....
46       | 47
\end{verbatim}
\begin{verbatim}
..xxxx.. ..xxxx.. | ..xxxx.. | ..xxxx.. ..xxxx.. ..xxxx.. ..xxxx.. | ..xxxx..
.xxxxxx. .xxxxxx. | .xxxxxx. | .xxxxxx. .xxxxxx. .xxxxxx. .xxxxxx. | .xxxxxx.
xxxxxxx. xxxxxxxx | xxxxxxxx | xxxxxxxx xxxxxxxx xxxxxxxx xxxxxxx. | xxxxxxxx
xxxxxxxx xxxxxxxx | xxxxxxxx | xxxxxxxx xxxxxxxx xxxxxxxx xxxxxxxx | xxxxxxxx
xxxxxxxx xxxxxxxx | xxxxxxxx | xxxxxxxx xxxxxxxx xxxxxxxx xxxxxxxx | xxxxxxxx
xxxxxxx. .xxxxxx. | xxxxxxx. | xxxxxxx. xxxxxxxx xxxxxxx. .xxxxxxx | xxxxxxxx
.xxxxxx. .xxxxxx. | .xxxxxx. | .xxxxxx. .xxxxxx. .xxxxxx. .xxxxxx. | .xxxxxx.
...xx... ...xx... | ...xx... | ...xxx.. ...xx... ..xxx... ..xxxx.. | ..xxx...
48                | 49       | 50                                  | 51
\end{verbatim}
\begin{verbatim}
..xxxx.. | .xxxxx.. | .xxxxx.. .xxxxxx. | .xxxxxx. | ..xxxxx.. .xxxxxx.
.xxxxxx. | .xxxxxx. | xxxxxxx. .xxxxxx. | xxxxxxx. | .xxxxxxx. xxxxxxx.
xxxxxxxx | xxxxxxxx | xxxxxxxx xxxxxxxx | xxxxxxxx | xxxxxxxx. xxxxxxxx
xxxxxxxx | xxxxxxxx | xxxxxxxx xxxxxxxx | xxxxxxxx | xxxxxxxx. xxxxxxxx
xxxxxxxx | xxxxxxxx | xxxxxxxx xxxxxxxx | xxxxxxxx | xxxxxxxxx xxxxxxxx
xxxxxxxx | xxxxxxxx | xxxxxxxx xxxxxxxx | xxxxxxxx | xxxxxxxx. xxxxxxxx
.xxxxxx. | .xxxxxx. | .xxxxxx. .xxxxxx. | .xxxxxx. | .xxxxxxx. xxxxxxx.
..xxxx.. | ..xxxx.. | ..xxxx.. ..xxxx.. | ..xxxx.. | ..xxxx... ..xxxx..
52       | 53       | 54                | 55       | 56 
\end{verbatim}

\medskip

\begin{verbatim}
..xxxxx.. .xxxxxx. ..xxxxx.. | ..xxxxx.. ..xxxxx.. ..xxxxx.. | ..xxxxx..
.xxxxxxx. xxxxxxxx .xxxxxxx. | .xxxxxxx. .xxxxxxx. .xxxxxxx. | .xxxxxxx.
xxxxxxxx. xxxxxxxx .xxxxxxx. | xxxxxxxx. xxxxxxxx. xxxxxxxx. | xxxxxxxx.
xxxxxxxxx xxxxxxxx xxxxxxxxx | xxxxxxxxx xxxxxxxxx xxxxxxxxx | xxxxxxxxx
xxxxxxxx. xxxxxxxx xxxxxxxxx | xxxxxxxx. xxxxxxxxx xxxxxxxxx | xxxxxxxxx
xxxxxxxx. xxxxxxxx .xxxxxxx. | xxxxxxxx. .xxxxxxx. xxxxxxxx. | xxxxxxxx.
.xxxxxxx. .xxxxxx. .xxxxxxx. | .xxxxxxx. .xxxxxxx. .xxxxxxx. | .xxxxxxx.
..xxxx... ..xxxx.. ..xxxxx.. | ..xxxxx.. ..xxxxx.. ..xxxx... | ..xxxxx..
56                           | 57                            | 58
\end{verbatim}
\begin{verbatim}
          |           | ..xxxxx.. | ..xxxxx.. ..xxxx... ..xxxxx.. ..xxxxx..
..xxxxx.. | ..xxxxx.. | .xxxxxxx. | .xxxxxxx. .xxxxxxx. .xxxxxxx. .xxxxxxx.
.xxxxxxx. | .xxxxxxx. | xxxxxxxxx | xxxxxxxxx xxxxxxxx. xxxxxxxx. xxxxxxxx.
xxxxxxxxx | xxxxxxxxx | xxxxxxxxx | xxxxxxxxx xxxxxxxxx xxxxxxxxx xxxxxxxxx
xxxxxxxxx | xxxxxxxxx | xxxxxxxxx | xxxxxxxxx xxxxxxxxx xxxxxxxxx xxxxxxxxx
xxxxxxxxx | xxxxxxxxx | xxxxxxxxx | xxxxxxxxx xxxxxxxxx xxxxxxxxx xxxxxxxxx
xxxxxxxx. | xxxxxxxxx | .xxxxxxx. | .xxxxxxx. .xxxxxxx. .xxxxxxx. .xxxxxxx.
.xxxxxxx. | .xxxxxxx. | ..xxxxx.. | .xxxxxx.. .xxxxxx.. .xxxxxx.. .xxxxxx..
..xxxxx.. | ..xxxxx.. | ....x.... | ....x.... ...xxx... ...xx.... ....xx...
59        | 60        | 61        | 62
\end{verbatim}
\begin{verbatim}
..xxxxx.. ..xxxx... ..xxxxx.. ..xxxxx.. ...xxxx... | ..xxxxx.. ..xxxx...
.xxxxxxx. .xxxxxxx. .xxxxxxx. .xxxxxxx. ..xxxxxx.. | .xxxxxxx. .xxxxxxx.
.xxxxxxx. .xxxxxxx. xxxxxxxxx xxxxxxxx. .xxxxxxxx. | xxxxxxxx. xxxxxxxx.
xxxxxxxxx xxxxxxxxx xxxxxxxxx xxxxxxxxx .xxxxxxxx. | xxxxxxxxx xxxxxxxxx
xxxxxxxxx xxxxxxxxx xxxxxxxxx xxxxxxxxx xxxxxxxxxx | xxxxxxxxx xxxxxxxxx
xxxxxxxxx xxxxxxxxx xxxxxxxxx xxxxxxxxx .xxxxxxxx. | xxxxxxxxx xxxxxxxxx
.xxxxxxx. .xxxxxxx. .xxxxxxx. .xxxxxxx. .xxxxxxxx. | .xxxxxxx. .xxxxxxx. 
.xxxxxx.. .xxxxxxx. ..xxxxx.. ..xxxxx.. ..xxxxxx.. | .xxxxxx.. .xxxxxxx.
...xxx... ...xxx... ...xx.... ...xxx... ...xxxx... | ...xxx... ...xxx...
62                                                 | 63
\end{verbatim}
\begin{verbatim}
..xxxxx.. | ..xxxxx.. ..xxxxx.. | ..xxxxx.. ..xxxxx.. | ..xxxxx.. | ..xxxxx..
.xxxxxxx. | .xxxxxxx. .xxxxxxx. | .xxxxxxx. .xxxxxxx. | .xxxxxxx. | .xxxxxxx.
.xxxxxxx. | xxxxxxxx. xxxxxxxx. | xxxxxxxxx xxxxxxxx. | xxxxxxxxx | xxxxxxxxx
xxxxxxxxx | xxxxxxxxx xxxxxxxxx | xxxxxxxxx xxxxxxxxx | xxxxxxxxx | xxxxxxxxx
xxxxxxxxx | xxxxxxxxx xxxxxxxxx | xxxxxxxxx xxxxxxxxx | xxxxxxxxx | xxxxxxxxx
xxxxxxxxx | xxxxxxxxx xxxxxxxxx | xxxxxxxxx xxxxxxxxx | xxxxxxxxx | xxxxxxxxx
.xxxxxxx. | .xxxxxxx. xxxxxxxx. | .xxxxxxx. xxxxxxxx. | xxxxxxxx. | xxxxxxxx.
.xxxxxxx. | .xxxxxxx. .xxxxxx.. | .xxxxxxx. .xxxxxxx. | .xxxxxxx. | .xxxxxxx.
...xxx... | ...xxx... ...xxx... | ...xxx... ...xxx... | ...xxx... | ...xxxx..
63        | 64                  | 65                  | 66        | 67 
\end{verbatim}
\begin{verbatim}
..xxxxx.. ..xxxxx.. ..xxxxx.. | ..xxxxx.. | ..xxxxx.. | ..xxxxx... .xxxxxx..
.xxxxxxx. .xxxxxxx. .xxxxxxx. | .xxxxxxx. | .xxxxxxx. | .xxxxxxx.. .xxxxxxx.
xxxxxxxxx xxxxxxxxx xxxxxxxx. | xxxxxxxxx | xxxxxxxxx | xxxxxxxxx. xxxxxxxxx
xxxxxxxxx xxxxxxxxx xxxxxxxxx | xxxxxxxxx | xxxxxxxxx | xxxxxxxxx. xxxxxxxxx
xxxxxxxxx xxxxxxxxx xxxxxxxxx | xxxxxxxxx | xxxxxxxxx | xxxxxxxxxx xxxxxxxxx
xxxxxxxxx xxxxxxxxx xxxxxxxxx | xxxxxxxxx | xxxxxxxxx | xxxxxxxxx. xxxxxxxxx
xxxxxxxxx xxxxxxxx. .xxxxxxxx | xxxxxxxxx | xxxxxxxxx | xxxxxxxxx. xxxxxxxxx
.xxxxxxx. .xxxxxxx. .xxxxxxx. | .xxxxxxx. | .xxxxxxx. | .xxxxxxx.. .xxxxxxx.
...xxx... ..xxxx... ..xxxxx.. | ..xxxx... | ..xxxxx.. | ..xxxxx... ..xxxxx..
67                            | 68        | 69        | 70
\end{verbatim}
\begin{verbatim}
..xxxxx... ..xxxxxx.. .xxxxxx.. .xxxxxxx. | ..xxxxx... ..xxxxxx.. | ..xxxxxx..
.xxxxxxxx. .xxxxxxx.. xxxxxxxx. .xxxxxxx. | .xxxxxxxx. .xxxxxxxx. | .xxxxxxxx.
xxxxxxxxx. xxxxxxxxx. xxxxxxxxx xxxxxxxxx | xxxxxxxxx. xxxxxxxxx. | xxxxxxxxx.
xxxxxxxxx. xxxxxxxxx. xxxxxxxxx xxxxxxxxx | xxxxxxxxx. xxxxxxxxx. | xxxxxxxxx.
xxxxxxxxxx xxxxxxxxxx xxxxxxxxx xxxxxxxxx | xxxxxxxxxx xxxxxxxxxx | xxxxxxxxxx
xxxxxxxxx. xxxxxxxxx. xxxxxxxxx xxxxxxxxx | xxxxxxxxx. xxxxxxxxx. | xxxxxxxxxx
xxxxxxxxx. xxxxxxxxx. xxxxxxxxx xxxxxxxxx | xxxxxxxxx. xxxxxxxxx. | xxxxxxxxx.
.xxxxxxx.. .xxxxxxx.. .xxxxxxx. .xxxxxxx. | .xxxxxxxx. .xxxxxxx.. | .xxxxxxx.. 
..xxxxx... ..xxxxx... ..xxxxx.. ..xxxxx.. | ..xxxxx... ..xxxxx... | ..xxxxx...
71                                        | 72                    | 73 
\end{verbatim}
\begin{verbatim}
                      ..xxxxx... ..xxxxx...                       ...xxxx... 
..xxxxx... ..xxxxxx.. .xxxxxxx.. .xxxxxxx.. ..xxxxxx.. ..xxxxxx.. .xxxxxxx..
.xxxxxxxx. .xxxxxxxx. xxxxxxxxx. xxxxxxxxx. .xxxxxxxx. .xxxxxxxx. .xxxxxxxx.
xxxxxxxxx. xxxxxxxxx. xxxxxxxxx. xxxxxxxxx. xxxxxxxxx. .xxxxxxxx. xxxxxxxxx.
xxxxxxxxxx xxxxxxxxxx xxxxxxxxxx xxxxxxxxxx xxxxxxxxx. xxxxxxxxxx xxxxxxxxxx
xxxxxxxxxx xxxxxxxxxx xxxxxxxxx. xxxxxxxxx. xxxxxxxxxx xxxxxxxxxx xxxxxxxxxx
xxxxxxxxx. xxxxxxxxx. xxxxxxxxx. xxxxxxxxx. xxxxxxxxx. xxxxxxxxxx xxxxxxxxx.
xxxxxxxxx. xxxxxxxxx. .xxxxxxxx. .xxxxxxxx. xxxxxxxxx. .xxxxxxxx. .xxxxxxxx.
.xxxxxxxx. .xxxxxxx.. ..xxxxxx.. ..xxxxxx.. .xxxxxxxx. .xxxxxxxx. ..xxxxxx..
..xxxxx... ..xxxxx... ....x..... .....x.... ..xxxxx... ..xxxxx... ....xx....
73 
\end{verbatim}
\begin{verbatim}
           |            ..xxxxx...            |            ...xxxx...
..xxxxxx.. | ..xxxxxx.. .xxxxxxx.. ..xxxxxx.. | ..xxxxxx.. ..xxxxxx..
.xxxxxxxx. | .xxxxxxxx. xxxxxxxxx. .xxxxxxxx. | .xxxxxxxx. .xxxxxxxx.
.xxxxxxxx. | xxxxxxxxx. xxxxxxxxx. xxxxxxxxx. | xxxxxxxxx. xxxxxxxxxx
xxxxxxxxxx | xxxxxxxxxx xxxxxxxxxx xxxxxxxxxx | xxxxxxxxxx xxxxxxxxxx
xxxxxxxxxx | xxxxxxxxxx xxxxxxxxxx xxxxxxxxxx | xxxxxxxxxx xxxxxxxxxx
xxxxxxxxxx | xxxxxxxxxx xxxxxxxxx. xxxxxxxxxx | xxxxxxxxxx xxxxxxxxxx
.xxxxxxxx. | xxxxxxxxx. .xxxxxxxx. .xxxxxxxx. | xxxxxxxxx. .xxxxxxxx.
.xxxxxxxx. | .xxxxxxxx. ..xxxxxx.. .xxxxxxxx. | .xxxxxxxx. ..xxxxxx..
..xxxxxx.. | ..xxxxx... ....xx.... ..xxxxxx.. | ..xxxxxx.. ...xxxx...
74         | 75                               | 76 
\end{verbatim}
\begin{verbatim}
...xxxx... | ...xxxx... ..xxxxxx.. ...xxxx... | ..xxxxxx.. ..xxxxxx.. ..xxxxx... 
.xxxxxxx.. | .xxxxxxxx. .xxxxxxxx. .xxxxxxx.. | .xxxxxxxx. .xxxxxxxx. .xxxxxxxx.
.xxxxxxxx. | .xxxxxxxx. .xxxxxxxx. .xxxxxxxx. | xxxxxxxxx. .xxxxxxxx. .xxxxxxxx. 
xxxxxxxxxx | xxxxxxxxxx xxxxxxxxxx xxxxxxxxxx | xxxxxxxxxx xxxxxxxxxx xxxxxxxxxx 
xxxxxxxxxx | xxxxxxxxxx xxxxxxxxxx xxxxxxxxxx | xxxxxxxxxx xxxxxxxxxx xxxxxxxxxx 
xxxxxxxxxx | xxxxxxxxxx xxxxxxxxxx xxxxxxxxxx | xxxxxxxxxx xxxxxxxxxx xxxxxxxxxx 
xxxxxxxxxx | xxxxxxxxxx xxxxxxxxxx xxxxxxxxxx | xxxxxxxxxx xxxxxxxxxx xxxxxxxxxx 
.xxxxxxxx. | .xxxxxxxx. .xxxxxxxx. .xxxxxxxx. | .xxxxxxxx. .xxxxxxxx. .xxxxxxxx. 
..xxxxxx.. | ..xxxxxx.. ..xxxxxx.. ..xxxxxxx. | ..xxxxxx.. ..xxxxxx.. ..xxxxxx.. 
...xxxx... | ...xxxx... ....xx.... ...xxxx... | ....xx.... ...xxx.... ...xxxx... 
77         | 78                               | 79 
\end{verbatim} 

\pagebreak

\begin{verbatim}
...xxxx... | ...xxxx... ..xxxxxx.. | ..xxxxx... ..xxxxxx.. ..xxxxxx.. ..xxxxx...
.xxxxxxxx. | .xxxxxxxx. .xxxxxxxx. | .xxxxxxxx. .xxxxxxxx. .xxxxxxxx. .xxxxxxxx.
.xxxxxxxx. | .xxxxxxxx. .xxxxxxxx. | xxxxxxxxx. xxxxxxxxx. .xxxxxxxx. .xxxxxxxx.
xxxxxxxxxx | xxxxxxxxxx xxxxxxxxxx | xxxxxxxxxx xxxxxxxxxx xxxxxxxxxx xxxxxxxxxx
xxxxxxxxxx | xxxxxxxxxx xxxxxxxxxx | xxxxxxxxxx xxxxxxxxxx xxxxxxxxxx xxxxxxxxxx
xxxxxxxxxx | xxxxxxxxxx xxxxxxxxxx | xxxxxxxxxx xxxxxxxxxx xxxxxxxxxx xxxxxxxxxx
xxxxxxxxxx | xxxxxxxxxx xxxxxxxxxx | xxxxxxxxxx xxxxxxxxxx xxxxxxxxxx xxxxxxxxxx
.xxxxxxxx. | .xxxxxxxx. .xxxxxxxx. | .xxxxxxxx. .xxxxxxxx. .xxxxxxxx. .xxxxxxxx.
.xxxxxxx.. | .xxxxxxxx. ..xxxxxx.. | .xxxxxxx.. ..xxxxxx.. .xxxxxxx.. .xxxxxxxx. 
...xxxx... | ...xxxx... ...xxxx... | ...xxxx... ...xxxx... ...xxxx... ...xxxx...
79         | 80                    | 81
\end{verbatim} 
\begin{verbatim}
..xxxxxx.. ..xxxxx... ..xxxxxx.. | ..xxxxxx.. ..xxxxxx.. | ..xxxxxx.. ..xxxxxx..
.xxxxxxxx. .xxxxxxxx. .xxxxxxxx. | .xxxxxxxx. .xxxxxxxx. | .xxxxxxxx. .xxxxxxxx.
xxxxxxxxx. xxxxxxxxx. .xxxxxxxx. | xxxxxxxxx. xxxxxxxxx. | xxxxxxxxxx xxxxxxxxx.
xxxxxxxxxx xxxxxxxxxx xxxxxxxxxx | xxxxxxxxxx xxxxxxxxxx | xxxxxxxxxx xxxxxxxxxx
xxxxxxxxxx xxxxxxxxxx xxxxxxxxxx | xxxxxxxxxx xxxxxxxxxx | xxxxxxxxxx xxxxxxxxxx
xxxxxxxxxx xxxxxxxxxx xxxxxxxxxx | xxxxxxxxxx xxxxxxxxxx | xxxxxxxxxx xxxxxxxxxx
xxxxxxxxxx xxxxxxxxxx xxxxxxxxxx | xxxxxxxxxx xxxxxxxxxx | xxxxxxxxxx xxxxxxxxxx
.xxxxxxxx. .xxxxxxxx. .xxxxxxxx. | .xxxxxxxx. xxxxxxxxx. | .xxxxxxxx. xxxxxxxxx.
.xxxxxxx.. .xxxxxxxx. .xxxxxxxx. | .xxxxxxxx. .xxxxxxx.. | .xxxxxxxx. .xxxxxxxx.
...xxxx... ...xxxx... ...xxxx... | ...xxxx... ...xxxx... | ...xxxx... ...xxxx...
82                               | 83                    | 84
\end{verbatim} 
\begin{verbatim}
..xxxxxx.. | ..xxxxxx.. ..xxxxxx.. ..xxxxxx.. ..xxxxxx.. | ..xxxxxx.. | ..xxxxxx..
.xxxxxxxx. | .xxxxxxxx. .xxxxxxxx. .xxxxxxxx. .xxxxxxxx. | .xxxxxxxx. | .xxxxxxxx.
xxxxxxxxxx | xxxxxxxxxx xxxxxxxxxx xxxxxxxxx. xxxxxxxxxx | xxxxxxxxxx | xxxxxxxxxx
xxxxxxxxxx | xxxxxxxxxx xxxxxxxxxx xxxxxxxxxx xxxxxxxxxx | xxxxxxxxxx | xxxxxxxxxx
xxxxxxxxxx | xxxxxxxxxx xxxxxxxxxx xxxxxxxxxx xxxxxxxxxx | xxxxxxxxxx | xxxxxxxxxx
xxxxxxxxxx | xxxxxxxxxx xxxxxxxxxx xxxxxxxxxx xxxxxxxxxx | xxxxxxxxxx | xxxxxxxxxx
xxxxxxxxxx | xxxxxxxxxx xxxxxxxxxx xxxxxxxxxx xxxxxxxxxx | xxxxxxxxxx | xxxxxxxxxx
xxxxxxxxx. | xxxxxxxxx. xxxxxxxxx. .xxxxxxxxx xxxxxxxxxx | xxxxxxxxxx | xxxxxxxxxx
.xxxxxxxx. | .xxxxxxxx. .xxxxxxxx. .xxxxxxxx. .xxxxxxxx. | .xxxxxxxx. | .xxxxxxxx.
...xxxx... | ...xxxxx.. ..xxxxx... ..xxxxxx.. ...xxxx... | ..xxxxx... | ..xxxxxx..
85         | 86                                          | 87         | 88 
\end{verbatim}
\begin{verbatim}
            | ..xxxxxx...             | ...xxxx.... ...xxxxx... | ...xxxxx...
..xxxxxx... | .xxxxxxxx.. ..xxxxxx... | ..xxxxxxx.. ..xxxxxxx.. | ..xxxxxxx..
.xxxxxxxx.. | xxxxxxxxxx. .xxxxxxxx.. | .xxxxxxxxx. .xxxxxxxxx. | .xxxxxxxxx.
xxxxxxxxxx. | xxxxxxxxxx. xxxxxxxxxx. | xxxxxxxxxx. .xxxxxxxxx. | xxxxxxxxxx.
xxxxxxxxxx. | xxxxxxxxxx. xxxxxxxxxx. | xxxxxxxxxxx xxxxxxxxxxx | xxxxxxxxxxx
xxxxxxxxxxx | xxxxxxxxxxx xxxxxxxxxxx | xxxxxxxxxxx xxxxxxxxxxx | xxxxxxxxxxx
xxxxxxxxxx. | xxxxxxxxxx. xxxxxxxxxxx | xxxxxxxxxxx xxxxxxxxxxx | xxxxxxxxxxx
xxxxxxxxxx. | xxxxxxxxxx. xxxxxxxxxx. | .xxxxxxxxx. .xxxxxxxxx. | .xxxxxxxxx.
xxxxxxxxxx. | .xxxxxxxx.. xxxxxxxxxx. | .xxxxxxxxx. .xxxxxxxxx. | .xxxxxxxxx.
.xxxxxxxx.. | ..xxxxxx... .xxxxxxxx.. | ..xxxxxxx.. ..xxxxxxx.. | ..xxxxxxx..
..xxxxxx... | .....x..... ..xxxxxx... | ....xxx.... ....xxx.... | ....xxx....
89          | 90                      | 91                      | 92
\end{verbatim}

\pagebreak

\begin{verbatim}
...xxxxx... ...xxxxx... | ...xxxxx... ...xxxxx... | ...xxxxx... ...xxxxx...
..xxxxxxx.. ..xxxxxxx.. | ..xxxxxxx.. .xxxxxxxx.. | .xxxxxxxx.. ..xxxxxxx..
.xxxxxxxxx. .xxxxxxxxx. | .xxxxxxxxx. .xxxxxxxxx. | .xxxxxxxxx. .xxxxxxxxx.
xxxxxxxxxxx xxxxxxxxxx. | xxxxxxxxxxx xxxxxxxxxx. | xxxxxxxxxxx xxxxxxxxxxx
xxxxxxxxxxx xxxxxxxxxxx | xxxxxxxxxxx xxxxxxxxxxx | xxxxxxxxxxx xxxxxxxxxxx
xxxxxxxxxxx xxxxxxxxxxx | xxxxxxxxxxx xxxxxxxxxxx | xxxxxxxxxxx xxxxxxxxxxx
xxxxxxxxxxx xxxxxxxxxxx | xxxxxxxxxxx xxxxxxxxxxx | xxxxxxxxxxx xxxxxxxxxxx
.xxxxxxxxx. xxxxxxxxxx. | xxxxxxxxxx. xxxxxxxxxx. | xxxxxxxxxx. xxxxxxxxxx.
.xxxxxxxxx. .xxxxxxxxx. | .xxxxxxxxx. .xxxxxxxxx. | .xxxxxxxxx. .xxxxxxxxx.
..xxxxxxx.. ..xxxxxxx.. | ..xxxxxxx.. ..xxxxxxx.. | ..xxxxxxx.. ..xxxxxxx..
....xxx.... ....xxx.... | ....xxx.... ....xxx.... | ....xxx.... ....xxxx...
93                      | 94                      | 95 
\end{verbatim}
\begin{verbatim}
...xxxxx... ...xxxxx... ...xxxxx... | ...xxxxx... | ...xxxxx... | ...xxxxx...
..xxxxxxx.. ..xxxxxxx.. ..xxxxxxx.. | ..xxxxxxx.. | ..xxxxxxx.. | .xxxxxxxx..
.xxxxxxxxx. .xxxxxxxxx. .xxxxxxxxx. | .xxxxxxxxx. | .xxxxxxxxx. | .xxxxxxxxx.
xxxxxxxxxxx xxxxxxxxxxx xxxxxxxxxx. | xxxxxxxxxxx | xxxxxxxxxxx | xxxxxxxxxxx
xxxxxxxxxxx xxxxxxxxxxx xxxxxxxxxxx | xxxxxxxxxxx | xxxxxxxxxxx | xxxxxxxxxxx
xxxxxxxxxxx xxxxxxxxxxx xxxxxxxxxxx | xxxxxxxxxxx | xxxxxxxxxxx | xxxxxxxxxxx
xxxxxxxxxxx xxxxxxxxxxx xxxxxxxxxxx | xxxxxxxxxxx | xxxxxxxxxxx | xxxxxxxxxxx
xxxxxxxxxx. xxxxxxxxxxx .xxxxxxxxxx | xxxxxxxxxxx | xxxxxxxxxxx | xxxxxxxxxxx
.xxxxxxxxx. .xxxxxxxxx. .xxxxxxxxx. | .xxxxxxxxx. | .xxxxxxxxx. | .xxxxxxxxx.
..xxxxxxx.. ..xxxxxxx.. ..xxxxxxx.. | ..xxxxxxx.. | ..xxxxxxx.. | ..xxxxxxx..
...xxxx.... ....xxx.... ...xxxxx... | ...xxxx.... | ...xxxxx... | ...xxxxx...
95                                  | 96          | 97          | 98
\end{verbatim}
\begin{verbatim}
...xxxxx... ...xxxxx... | ..xxxxxx... ...xxxxx...
.xxxxxxxx.. .xxxxxxxxx. | .xxxxxxxxx. .xxxxxxxxx.
.xxxxxxxxx. .xxxxxxxxx. | .xxxxxxxxx. .xxxxxxxxx.
xxxxxxxxxxx xxxxxxxxxxx | xxxxxxxxxxx xxxxxxxxxxx
xxxxxxxxxxx xxxxxxxxxxx | xxxxxxxxxxx xxxxxxxxxxx
xxxxxxxxxxx xxxxxxxxxxx | xxxxxxxxxxx xxxxxxxxxxx
xxxxxxxxxxx xxxxxxxxxxx | xxxxxxxxxxx xxxxxxxxxxx
xxxxxxxxxxx xxxxxxxxxxx | xxxxxxxxxxx xxxxxxxxxxx
.xxxxxxxxx. .xxxxxxxxx. | .xxxxxxxxx. .xxxxxxxxx.
..xxxxxxxx. ..xxxxxxx.. | ..xxxxxxx.. .xxxxxxxx..
...xxxxx... ...xxxxx... | ...xxxxx... ...xxxxx...
99                      | 100
\end{verbatim}

\begin{table}[htp]
\begin{center}
\begin{tabular}{lllllllllllllllllll}
\hline
  $n$ &                   51 & 52 & 53 & 54 & 55 & 56 & 57 & 58 & 59 & 60 & 61 & 62 & 63 \\
  $S_{\square}(n)\ge$ &   226 & 237 & 245 & 254 & 263 & 272 & 282 & 293 & 303 & 314 & 324 & 334 & 346 \\ 
\hline
\end{tabular}
\begin{tabular}{lllllllllllllllllll}
\hline
  $n$ &                   64 & 65 & 66 & 67 & 68 & 69 & 70& 71 & 72 & 73 & 74 & 75 & 76  \\
  $S_{\square}(n)\ge$ &   358 & 370 & 382 & 394 & 407 & 421 & 431 & 442 & 454 & 466 & 480 & 493 & 507\\ 
\hline
\end{tabular}
\begin{tabular}{lllllllllllllllllll}
\hline
  $n$ &                   77 & 78 & 79 & 80 & 81 & 82 & 83 & 84 & 85 & 86 & 87 & 88 & 89 \\
  $S_{\square}(n)\ge$ &   521 & 535 & 549 & 564 & 578 & 593 & 608 & 623 & 638 & 653 & 669 & 686 & 700 \\ 
\hline
\end{tabular}
\begin{tabular}{lllllllllllllllllll}
\hline
  $n$ &                   90 & 91 & 92 & 93 & 94 & 95 & 96 & 97 & 98 & 99 & 100\\
  $S_{\square}(n)\ge$ &   715 & 731 & 748 & 765 & 782 & 799 & 817 & 836 & 853 & 870 & 887\\ 
\hline
\end{tabular}
\caption{Lower bounds for $S_{\square}(n)$ for $51\le n\le 100$.}
\label{table_lower_bounds_number_of_contained_squares2}
\end{center}
\end{table}

\section{Point sets with many squares such that no pair of points is contained in two different squares}
\label{appendix_hamming}

Point sets showing $S_{\square}\!\left(10,\cP_{6,2}^\star\right)\ge 4$, $S_{\square}\!\left(11,\cP_{6,2}^\star\right)\ge 5$, 
$S_{\square}\!\left(12,\cP_{6,2}^\star\right)\ge 6$, and $S_{\square}\!\left(13,\cP_{6,2}^\star\right)\ge 7$ are given by: 
\begin{verbatim}
                                                 x.xx..                 .x.x..    
              .xxx.         .xx..  .xx..  x..x.  ....x.  xxxx..  x.x..  ......
x.xx.  xx.x.  x.xx.  .x.x.  .xx..  x.x..  .x...  .x...x  ....x.  .xx.x  xx.x..
..xx.  xxx..  .x...  xx.x.  x.x.x  .x.xx  ...xx  x..x..  .x...x  xxx..  xx...x
x.x.x  ....x  ....x  xx.xx  .x...  .x...  x.xxx  ...x..  x..x..  ...x.  ....x.
x..x.  x.xx.  .xx..  ..x..  ..xx.  .x.x.  ..x..  ...x..  ...x..  .x...  ..x...
10 
\end{verbatim} 
\begin{verbatim}
                                 |                 | .xxxx..
        xxxx..   .x.xx..         | .xxxx..         | .....x. 
x.xx.   ....x.   .....x.   x.x.. | .....x.   .xx.. | ..x...x
..xx.   .x...x   ..x...x   .xx.x | ..x...x   .xxx. | .x..x..
x.x.x   x..x..   .x..x..   xxx.. | .x..x..   x...x | x...x..
x..x.   ...x..   x...x..   ...x. | x...x..   .xxx. | ....x..
.x...   ...x..   ....x..   .x..x | ....x..   .xx.. | ....x..
11                               | 12              | 13
\end{verbatim} 

\section{Point sets with many squares determining the same directions}
In the introduction we mentioned the problem of the maximum number $S_{\Vert\,  \square}(n)$ of axis-parallel squares spanned by $n$ points. We may 
also say that the counted squares determine the same directions. Looping over the point sets that can be obtained by recursive $2$-extension gives the lower 
bounds $S_{\Vert\,  \square}(4)\ge 1$, $S_{\Vert\,  \square}(6)\ge 2$, $S_{\Vert\,  \square}(7)\ge 2$, $S_{\Vert\,  \square}(8)\ge 3$, $S_{\Vert\,  \square}(9)\ge 5$, 
$S_{\Vert\,  \square}(10)\ge 5$, $S_{\Vert\,  \square}(11)\ge 6$, $S_{\Vert\,  \square}(12)\ge 8$, $S_{\Vert\,  \square}(13)\ge 8$, $S_{\Vert\,  \square}(14)\ge 9$, 
$S_{\Vert\,  \square}(15)\ge 11$, $S_{\Vert\,  \square}(16)\ge 14$, and $S_{\Vert\,  \square}(17)\ge 14$: 
\begin{verbatim}
   |     | xx. xxx |      xxx xx.x xxx | xxx | xxx. | xxxx xxx.x xxxx
xx | xxx | xxx xxx | xxxx xxx xx.. xxx | xxx | xxxx | xxx. xxx.. xxxx
xx | xxx | .xx .x. | xxxx x.x xx.x xx. | xxx | xxx. | xxxx xxx.x xxx.
4  | 6   | 7       | 8                 | 9   | 10   | 11
\end{verbatim} 
\begin{verbatim}
     |       xxxx |              xxxx xxxx xxx.x xxx.x xxx..x       xxxx
xxxx | xxxx. xxxx | xxxxx xxxx.x xxx. xxxx xxx.. xxx.. xxx... xxxxx xxxx
xxxx | xxxxx xxxx | xxxx. xxxx.. xxxx xxxx xxx.. xxx.x xxx... xxxxx xxx.
xxxx | xxxx. .x.. | xxxxx xxxx.x x.xx x..x xxx.x xxx.. xxx..x xxxx. xx.x
12   | 13         | 14
\end{verbatim}
\begin{verbatim}
xxx.x                     | xxx           |      | ..x.
xxx.. xxxx xxxx xxx. xxxx | xxx x.xx .xxx | xxxx | xxxx
xxx.x xxxx xxxx xxxx xxxx | xxx xxxx xxxx | xxxx | xxxx
..... xxxx xxxx xxxx xxxx | xxx xxxx xxxx | xxxx | xxxx
x.x.x x.x. xx.. .xxx .xx. | xxx xxxx xxxx | xxxx | xxxx
14                        | 15            | 16   | 17
\end{verbatim}

\end{document}